\documentclass[11pt,oneside,reqno]{amsart}
\usepackage{amsmath,amssymb,amsbsy,amsfonts,amsthm,
latexsym,amsopn,amstext, amsxtra,euscript,amscd,hyperref,enumitem}
\usepackage{xcolor}
\newtheorem{theorem}{Theorem}[section]
\newtheorem{definition}[theorem]{Definition}
\newtheorem{prop}[theorem]{Proposition}
\newtheorem{corollary}[theorem]{Corollary}
\newtheorem{remark}[theorem]{Remark}
\newtheorem{example}[theorem]{Example}
\newtheorem{lemma}[theorem]{Lemma}

\def \C{\mathbb{C}}

\def \Q{\mathbb{Q}}
\def \R{\mathbb{R}}
\def \Z{\mathbb{Z}}

\begin{document}

\title[Period length of continued fractions]{$S$-units and period length of continued fractions of linear recursions}

\author{Veekesh Kumar, Vivek Singh and Johannes Sprang}

\pagenumbering{arabic}
\begin{abstract}
Let  $(A_n)_{n\in \Z}$ be a  linear recurrence sequence with values in a real quadratic field. In this paper, we study the question whether the period length of the continued fraction of $A_n$ is bounded as $n$ varies. The case where $(A_n)_n$ is a linear recurrence of degree $1$ has previously been solved by Corvaja and Zannier. Their result settled a problem posed by Mend\`es France about the length of the periods of the continued fractions for $\alpha^n$.
\end{abstract}

\address[Veekesh Kumar and Vivek Singh]{Department of Mathematics, Indian Institute of Technology Dharwad, Chikkamalligawad village, Dharwad, Karnataka - 580011, India.}
\address[Johannes Sprang]{ University of Duisburg-Essen, 45127 Essen, Germany.}

\email[]{veekeshk@iitdh.ac.in}
\email[]{MA23DP003@iitdh.ac.in}
\email[]{johannes.sprang@uni-due.de}

\subjclass[2010] {Primary 11J68, 11J87;  Secondary 11B37, 11R06 }
\keywords{Approximation to algebraic numbers,  Continued fraction, Period length, Power sums}

\maketitle

\section{Introduction}
It is well known that the continued fraction expansion of a quadratic irrational number is eventually periodic. For a quadratic number $\alpha$, we will use the notation $\ell(\alpha)$ to denote the period length of the continued fraction of $\alpha$ throughout the paper with the convention that $\ell(\alpha)=0$ if $\alpha$ is rational.  Given a parameterized, infinite family of quadratic irrational numbers such as a linear recurrence sequence of quadratic irrational numbers, it is, in general, very challenging to predict the lengths of the periods of their continued fraction expansions. In this context, first it was studied by Schinzel \cite{sch}, who proved that, if $P(x)$  is a non-constant polynomial with integer coefficients and positive leading term satisfying certain assumptions, then the sequence of period lengths of the continued fraction expansion of $\sqrt{P(n)}$ is unbounded. Later, in 2004, Corvaja and Zannier \cite{corv} applied the Schmidt Subspace Theorem to classify the quadratic irrational numbers for which the corresponding sequence of period lengths of $\alpha^n$ is either bounded, or unbounded which settle the problem posed by Mend\`es France \cite[Problem 6]{Me}. More precisely, they had shown that if $\alpha$ is not a unit in the ring of integer of $\mathbb{Q}(\alpha)$, then the sequence $(\ell(\alpha^n))$ is unbounded, otherwise it is bounded. By regarding $\{\alpha^n: n\in\mathbb{N}\}$  as special instance of linear recurrence sequences, it is natural to ask whether these results can be extended to an arbitrary linear recurrence sequence. In this paper, we study the question of unboundedness of the period lengths of linear recurrence sequences in real quadratic fields.
\smallskip

Let $K\subseteq \R$ be a real quadratic field and denote by $(\cdot)^\prime\colon K\xrightarrow{\sim} K$ the non-trivial automorphism of $K$. We consider linear recurrence sequences $\underline{A}=(A_n)_{n\in \Z}\in K^{\Z}$ of degree $k$ satisfying a linear recurrence equation
\[
	A_n=c_1 A_{n-1}+ \dots + c_kA_{n-k},\quad \text{ with }c_1,\dots,c_k\in K \text{ and }c_k\neq 0,
\]
and we write $P_{\underline{A}}$ for the characteristic polynomial of $\underline{A}$
\[
	P_{\underline{A}}:=X^k-c_1X^{k-1}-\cdots -c_k \in K[X].
\]
We call  roots of the characteristic polynomial $P_{\underline{A}}$ of a linear recurrence sequence $\underline{A}$ \emph{the roots of $\underline{A}$}. We write $\overline{\Q}$ for the algebraic closure of $\Q$ in $\C$ and view the roots of $P_{\underline{A}}$ as complex (algebraic) numbers. We call $\underline{A}$ \emph{non-degenerate} if for any roots $\alpha$ and $\beta$ of ${\underline{A}}$ we have
\[
	\frac{\alpha}{\beta} \text{ is a root of unity} \Leftrightarrow \alpha=\beta.
\]
Given a degenerate sequence $\underline{A}$, we can always find a positive integer $d$ such that all the subsequences
\[
	(A_{dn})_n, (A_{dn+1})_n,\dots, (A_{dn+(d-1)})_n
\]
are non-degenerate and the question of whether the sequence of period lengths $(\ell(A_n))_n$ is bounded can be decided by looking at these finitely many subsequences,  hence the assumption that $\underline{A}$ is non-degenerate is not restrictive.
To state the main theorems, we have to introduce some notation:
\begin{definition}
A polynomial $P\in K[X]$ is called \emph{unital} if all of its roots are units (i.e. units in the ring of integers of some finite extension). A linear recurrence sequence $\underline{A}$ is called unital if its characteristic polynomial $P_{\underline{A}}$ is unital.
\end{definition}

\begin{definition}
	We call a monic polynomial $P\in K[X]$ \emph{Pisot} if for every normalized irreducible factor  $\pi\in K[X]$ of $P$ we have $\pi\neq \pi'$ and one of the following conditions holds:
	\begin{enumerate}
	\item all roots of $\pi$ in $\C$ have absolute value $<1$ and all roots of $\pi'$ in $\C$ have absolute value $>1$.
	\item all roots of $\pi$ in $\C$ have absolute value $>1$ and all roots of $\pi'$ in $\C$ have absolute value $<1$.
	\end{enumerate}
	We call $P$ \emph{unital Pisot} if $P$ is unital and Pisot. A linear recurrence sequence $(A_n)_n$ with values in $K$ is called \emph{unital Pisot} if its characteristic polynomial is unital Pisot.
\end{definition}


We can now state the main theorem:
\begin{theorem}\label{thm:A}
Let $K$ be a real quadratic field and $(A_n)_{n\in\Z}$ a non-degenerate linear recurrence in $K$. 
If the sequence $(\ell(A_n))_n$ of period lengths is bounded, then $(A_n)_n$ belongs to one of the following three disjoint classes of linear recurrence sequences:
\begin{enumerate}
\item $(A_n)_n$ with $A_n\in \Q$ for all $n\in \Z$, 
\item $(A_n)_n$ a linear recurrence sequence such that $(A_n-A_n')_n$ is non-zero and essentially constant, i.e. $A_n-A_n'=(\pm 1)^n\beta$ for some $\beta\in K\setminus \Q$, and the linear recurrence sequence $(A_n+A_n')_n$ of rational numbers has bounded denominators.
\item $(A_n)_n$ a linear reccurence sequence with $(A_n-A_n')_n$ a non-zero linear recurrence sequence which is unital Pisot, and $(A_n+A_n')_n$ a linear recurrence in $\Q$ such that all roots of $(A_n+A_n')_n$ with absolute value $\geq 1$ are algebraic integers while all roots of absolute value $\leq 1$ are reciprocals of algebraic integers.
\end{enumerate} 
\end{theorem}

\begin{remark}
Let us make a few observations about sequences of type $(a)$, $(b)$ and $(c)$:
\begin{enumerate}
    \item Each of the three disjoint classes contains sequences of bounded period length. Indeed, it is not difficult to see that all sequences of the form $(a)$ or $(b)$ have bounded period lengths. A simple example of a sequence of type $(c)$ is given by $A_n=\alpha^n$ with $\alpha$ a fundamental unit in the real quadratic field $K$. It has been shown by Corvaja--Zannier that $(A_n)_n=(\alpha^n)_n$ has bounded period length. We provide further examples of sequences of type $(c)$ with bounded period length in section \ref{sec:unital-pisot}. However, we do not expect that all sequences of type $(c)$ have bounded period lengths. A systematic investigation of period lengths of sequences of type $(c)$ might be an interesting topic for future research.
    \item The condition on the rational part of sequences of type $(b)$ can be read of easily from the characteristic polynomial of $(A_n+A_n')_n$: $(A_n+A_n')_n$ has bounded denominators if and only if its characteristic polynomial is unital, see Corollary \ref{cor:nonArch-bounded-units}. 
\end{enumerate}
\end{remark}

The proof of this Theorem in the case $(A_n)_n=(\alpha^n)_n$ for $\alpha$ an element of the real quadratic field $K$ has been given by Corvaja--Zannier in \cite{corv}. We generalize their strategy to the case of arbitrary linear recurrence sequences. A key ingredient of their proof is a result on rational approximations of powers of an algebraic number based on Schmidt's Subspace Theorem, which has been generalized by Kulkarni--Mavraki--Nguyen to several variables, see \cite{kul}. In our approach, Schmidt's Subspace Theorem and the theory of $S$-unit equations play also an important role at several places, e.g. in the result of Kulkarni--Mavraki--Nguyen and in estimating the growth of linear recurrence sequences. An important difference to the case of Corvaja--Zannier is that we have to take into account the non-Archimedean absolute values of the linear recurrence sequences to get the full strength of Theorem \ref{thm:A}. 
For Archimedean absolute values the corresponding statement is a celebrated Theorem of Evertse and van der Poorten--Schlickewei, see \cite{Evertse,vdPSchlick}.
The corresponding result for non-Archimedean absolute values seems to be well-known to the experts, however we couldn't find any reference in the literature. In the appendix, we deduce this Theorem from a more general result of Evertse on $S$-unit equations.

\section{Continued fraction and period lengths}
We recall some basic facts about continued fractions. For an irrational real number $\alpha$, we will denote by
$
\alpha=[a_0(\alpha),a_1(\alpha),\ldots,a_n(\alpha),\ldots,]
$
the continued fraction expansion of $\alpha$, where $a_0(\alpha)\in \Z$ and $a_1(\alpha),a_2(\alpha),\ldots$ are positive integers.
We also write $\frac{p_n(\alpha)}{q_n(\alpha)}=[a_0(\alpha),a_1(\alpha),\ldots,a_n(\alpha)]$ be the $n$th convergent for $\alpha$ and it satisfies the inequality 
\begin{equation}\label{eq1.1}
\left|\alpha-\frac{p_n(\alpha)}{q_n(\alpha)}\right|\leq \frac{1}{a_{n+1}(\alpha)q^2_n(\alpha)}
\end{equation} 
for all $n\geq 1$. If $\alpha$ is clear from the context we will drop it from the notation and we will simply write $a_n$ and $p_n/q_n$ for the partial quotients and convergents of $\alpha$. 
\begin{definition}
 Let $\alpha\in \R$ be a quadratic irrational number and let $\alpha'$ be its conjugates.   We say $\alpha$ is  reduced if $\alpha>1$ and $-1<\alpha'<0$.  
\end{definition}
 It is known that continued fraction of an irrational real number $\alpha$ is eventually periodic if and only if it is quadratic irrational, and it is purely periodic if and only if $\alpha$ is reduced. 
\smallskip

We need the following basic lemma of continued fraction expansion (for the proof see \cite[Lemma 2.30]{BPSZ}).
\begin{lemma}\label{lemma1}
    Let $\alpha=[a_0;a_1,a_2,a_3,\cdots]>1$ and $\frac{p_n}{q_n}$ be the $n$th convergent. Then for all $n\geq 1$
    \begin{enumerate}
\item[(a)]  $a_1\cdots a_n \leq q_n\leq F_n a_1\cdots a_n$
\item[(b)] $a_0 a_1\cdots a_n \leq p_n\leq F_{n+1} a_0 a_1\cdots a_n$,
\end{enumerate}
where $F_n$ denotes the $n$th term of Fibonacci sequence, i.e $F_0=0,F_1=1$ and $F_{n+1}=F_n+F_{n-1}$ for $n\geq 1$.
\end{lemma}
The following lemma says that multiplying a given sequence of quadratic real numbers with a fixed positive integer does not affect the boundedness of the period length.
\begin{lemma}\label{lem:periodlength-mult-by-int}
\label{prop:integer}
    Let $K$ be a real quadratic field and $(A_n)$ be a sequence in $K$. Let $D$ be a positive integer.  Then the period length of $A_n$ is bounded if and only if the period length of $DA_n$ is bounded.
\end{lemma}
\begin{proof}
This follows from the more general fact that the period length of an irrational number under a Möbius transformation can be estimated as follows: Given a non-singular matrix $N=\begin{pmatrix}
a & b\\
c & d
\end{pmatrix}$ 
$\in M_2(\mathbb{Z})$, let $h_N:\mathbb{R}\ \backslash \ \{\frac{-d}{c}\} \to \mathbb{R}$ given by
\begin{equation*}
 h_N(x)=\frac{ax+b}{cx+d}   
\end{equation*}
be the associated M\"{o}bius transformation with the matrix $N$. 

By Theorem 1 in \cite{rada}, we have 
\begin{equation}\label{eq4}
\frac{1}{S_N} \ \ell(x)\leq \ell(h_N(x)) \leq S_N \  \ell(x) \ ,
\end{equation}
where $S_N$ is some fixed positive integer which depends only on the matrix $N$.  Applying this for $N=\begin{pmatrix}
D & 0\\
0 & 1
\end{pmatrix}$ proves the Lemma.
 
\end{proof}

The following proposition gives a criterion for a sequence of real quadratic irrationals having unbounded period length. 
\begin{prop}\label{prop1p}
Let $K$ be a real quadratic field and $(A_n)_{n\geq 0}$ be a  sequence of elements of $K$ with continued fraction expansion $A_n=[a_0(A_n); a_1(A_n),a_2(A_n),\ldots]$, which satisfies
	  \begin{equation}\label{eq:Condition_ai}
        \displaystyle\lim_{n\to\infty}\frac{\log a_i(A_n)}{n}=0
    	\end{equation}
        for all $i\neq 0$.
 Suppose furthermore that one of the following two conditions hold:
\begin{enumerate}
    \item We have
    \begin{equation}\label{eq:ConditionAnp-Arch}
		\lim_{n\to \infty} \frac{\log|A_n-A'_n|}{n}>0.
	\end{equation}
	\item There exists a non-Archimedean place $v$ such that 
	\begin{equation}\label{eq:ConditionAnp-nonArch}
		\lim_{n\to \infty} \frac{\log|A_n-A'_n|_\mathit{v}}{n}>0, \quad \text{ or } \quad\lim_{n\to \infty} \frac{\log|A_n|_\mathit{v}}{n}>0,
	\end{equation}
    and 
        \begin{equation}\label{eq:ConditionAnReduced-eps}
        |A_n-A'_n|\geq 2,
        \end{equation}
     for all $n$ sufficiently large.
\end{enumerate}
Then the sequence  $(\ell(A_n))_{n\geq 0}$ of period lengths is unbounded.
\end{prop}
\begin{proof}
After replacing $(A_n)_n$ by $(A_{n+N})_{n\geq 0}$ with $N$ sufficiently large, we may assume $|A_n-A'_n|\geq 2$ for all $n\geq 0$. This is part of condition $(b)$, while it is a consequence of  $\displaystyle\lim_{n\to \infty} \frac{\log|A_n-A'_n|}{n}>0$ in the case of condition $(a)$. Since neither conjugation nor multiplication by $(-1)$ in $K$, change the period length, we may  assume $A_n>A_n'$ for all $n\geq 0$. Let $k_n=\lceil A'_n\rceil$. Then  $x_n:=A_n-k_n>1$ and $-1<x_n'<0$ for all $n\geq 0$. Therefore, $x_n$ is reduced. We prove the Proposition by contradiction. Suppose the period length was bounded. We can find an integer $h$ such that the continued fraction expansion of $x_n$ is given for any $n\geq 0$ by, 
$$
x_n=[\overline{a_0(x_n);a_1(x_n),a_2(x_n),...,a_{h-1}(x_n)}].
$$
By assumption \eqref{eq:Condition_ai}, we have 
\begin{equation}\label{eq3.2}
\displaystyle\lim_{n\to\infty}\frac{\log a_i(x_n)}{n}=0
\end{equation}
for all $i\geq 0$.   
The algebraic numbers $x=x_n$ and $x=x_n'$ satisfy the following algebraic relation
\begin{align*}
x&=[a_0(x_n);a_1(x_n),a_2(x_n),\ldots,a_{h-1}(x_n),x]\\
&=\frac{p_{h-1}(x_n)x+p_{h-2}(x_n)}{q_{h-1}(x_n)x+q_{h-2}(x_n)}.
\end{align*}
Hence $x_n$ and $x_n'$ are the roots of the polynomial
\begin{equation}\label{eq:polynomial-of-xn}
X^2+\frac{q_{h-2}(x_n)-p_{h-1}(x_n)}{q_{h-1}(x_n)}X-\frac{p_{h-2}(x_n)}{q_{h-1}(x_n)}=0.
\end{equation}
By Lemma \ref{lemma1} and \eqref{eq3.2}, we deduce that 
\begin{equation}\label{eq3.3}
\displaystyle\lim_{n\to\infty}\frac{\log q_j(x_n)}{n}=0\quad \mbox{and} \displaystyle\lim_{n\to\infty}\frac{\log p_j(x_n)}{n}=0
\end{equation}
for all $j=0,1,\ldots,h-1$, and hence, the coefficients of this polynomial have logarithms that are o(n) for $n\to \infty$. 
\smallskip

\textit{(a): Suppose that hypothesis (a) holds}: The assumption
\begin{equation*}
	\lim_{n\to\infty}\frac{\log|x_n-x_n'|}{n}=\lim_{n\to\infty}\frac{\log|A_n-A_n'|}{n}>0,
\end{equation*}
together with the inequalities $x_n>1$ and $0<|x_n'|<1$ implies
\begin{equation*}
	\lim_{n\to\infty}\frac{\log|x_n|}{n}>0.
\end{equation*}
Using once more the inequality $0<|x_n'|<1$, we deduce
\begin{equation*}
	\lim_{n\to\infty}\frac{\log|x_n+x_n'|}{n}>0,
\end{equation*}
which contradicts the fact that the logarithm of the absolute value of the linear coefficient of \eqref{eq:polynomial-of-xn} is $o(n)$.
\smallskip

\textit{(b): Suppose that hypothesis (b) holds}: Let $v$ be a non-Archimedean place as in hypothesis (b).
Let us first make the following observation. Given any sequence $\left(\frac{c_n}{d_n}\right)_n$ of rational numbers with $c_n,d_n\in \mathbb{Z}$ satisfying $\displaystyle\lim_{n\to\infty} \frac{\log |c_n|}{n}=\lim_{n\to\infty} \frac{\log |d_n|}{n}=0$, then $\displaystyle\lim_{n\to\infty} \frac{\log \left|c_n/d_n\right|_v}{n}=0$. Indeed, this follows from $1\leq \frac{1}{|c_n|_v} \leq |c_n|$ and $1\leq \frac{1}{|d_n|_v} \leq |d_n|$. If we apply this observation to the coefficients of the minimal polynomial of $x_n$, then we obtain together with \eqref{eq3.3} 
\begin{equation}\label{eq:log-growth-of-coeff}
	\lim_{n\to\infty}\frac{\log \left| \frac{q_{h-2}(x_n)-p_{h-1}(x_n)}{q_{h-1}(x_n)}\right|_v}{n}=0\quad \mbox{and} \lim_{n\to\infty}\frac{\log \left| \frac{p_{h-2}(x_n)}{q_{h-1}(x_n)} \right|_v}{n}=0.
\end{equation}
Let us first prove that we may assume
\begin{equation}\label{eq:log-xn}
	\lim_{n\to \infty} \frac{\log|x_n|_v}{n}>0,
\end{equation}
after possibly replacing $v$ by $v'$. Indeed, if we are in the case
\[
	\lim_{n\to \infty} \frac{\log|A_n|_\mathit{v}}{n}>0
\]
of hypothesis $(b)$, then \eqref{eq:log-xn} follows immediately from the non-Archimedean triangle inequality. In the other case
\begin{equation*}
		\lim_{n\to \infty} \frac{\log|A_n-A'_n|_\mathit{v}}{n}>0,
\end{equation*}
we get
\begin{equation*}
	\lim_{n\to\infty}\frac{\log|x_n-x_n'|_v}{n}=\lim_{n\to\infty}\frac{\log|A_n-A_n'|_v}{n}>0,
\end{equation*}
and we deduce
\[
	\limsup_{n\to \infty} \frac{\log|x_n|_v}{n}>0 \text{ or } \limsup_{n\to \infty} \frac{\log|x_n'|_v}{n}>0.
\]
After possibly replacing $v$ with $v'$ and passing to a suitable sub-sequence, we get \eqref{eq:log-xn}. If we have
\[
	\limsup_{n\to \infty} \frac{\log|x_nx_n'|_v}{n}>0 \text{ or } \liminf_{n\to \infty} \frac{\log|x_nx_n'|_v}{n}<0
\]
we get a contradiction with \eqref{eq:log-growth-of-coeff}. Hence, we have
\begin{equation}\label{eq:log-xn-v}
	\lim_{n\to \infty} \frac{\log|x_n'|_v}{n}=-\lim_{n\to \infty} \frac{\log|x_n|_v}{n}<0.
\end{equation}
In particular, we have $|x_n|_v>1$ and $|x_n'|_v<1$ for all sufficiently large values of $n$. The non-Archimedean triangle inequality implies that, for all sufficiently large values of $n$, the equality
\[
	\frac{\log|x_n+x_n'|_v}{n}=\frac{\log|x_n|_v}{n}
\]
holds. Hence, by \eqref{eq:log-xn-v}, we get
\[
	\lim_{n\to \infty} \frac{\log|x_n+x_n'|_v}{n}>0
\]
which again contradicts \eqref{eq:log-growth-of-coeff}.
\end{proof}
We also require the following result of Schinzel \cite[Theorem 3]{sch} on the period length of the continued fraction to a square root of an integer-valued polynomial. 
\begin{theorem}\label{schinzel}
 Let $f(X)=a_0 X^p+\cdots+a_p$ be an integer-valued polynomial such that $f(n)\geq 0$ for all $n\geq 0$. If 
 \begin{enumerate}
     \item  $p\equiv 1~(\mathrm{mod}~2)$ or 
     \item  $p\equiv 0~(\mathrm{mod}~2)$ and $a_0$ is not a square of a rational number,
 \end{enumerate}
 then the sequence $\ell(\sqrt{f(n)})$ is unbounded.
\end{theorem}



\section{Linear recurrences and $S$-units}

In the following, we will collect some well-known results about linear recurrence sequences. Let $\underline{A}=(A_n)_n$ be a linear recurrence sequence in a number field $K\subseteq \overline{\Q}$. It is well-known that for each root $\alpha\in\overline{\Q}$ of $P_{\underline{A}}$ of multiplicity $m_\alpha$, there is a polynomial $p_{\alpha}(X;\underline{A})\in \overline{\Q}[X]$ of degree at most $m_\alpha-1$ such that
\[
	A_n=\sum_{\alpha \text{ root of } P_{\underline{A}}} p_{\alpha}(n;\underline{A})\alpha^n.
\]
We call this the \emph{power sum form} of the linear recurrence sequence $(A_n)_n$.
If $\underline{A}$ is clear from the context, we will simply write $p_\alpha(X)$ instead of $p_\alpha(X;\underline{A})$.
We have the following celebrated Theorem about the growth of linear recurrence sequences, see \cite{Evertse,EPSW,vdP,vdPSchlick}:
\begin{theorem}\label{EPS}
Let $\underline{A}=(A_n)_{n\geq 0}$ be a non-degenerate linear recurrence sequence in a number field $K$ and let $v$ a place of $K$. We order the roots $\alpha_1,\ldots,\alpha_k$ of the characteristic polynomial such that $|\alpha_1|_v\geq |\alpha_2|_v\geq \dots \geq |\alpha_k|_v$. Suppose that $|\alpha_1|_v>1$. Then for any $\epsilon>0$, we have
\[
	|A_n|_v\geq |\alpha_1|_v^{n(1-\epsilon)}
\]
for all sufficiently large $n$.
\end{theorem}

It is not difficult to deduce this Theorem from results of Evertse related to $S$-unit equations in \cite[Theorem 2]{Evertse}.
The proof of this Theorem is probably well-known to the experts but it is difficult to find a full proof in the literature. A proof in the Archimedean case of this Theorem has recently been written down by Fuchs and Heintze, see \cite{FH}. For the convenience of the reader, we work out the non-Arhimedean case in the appendix, see Theorem \ref{thm1}.

\begin{corollary}\label{cor:nonArch-bounded-units}
Let $(A_n)_{n\in \Z}$ be a linear recurrence sequence in a number field $K$. The following are equivalent:
\begin{enumerate}
    \item There exists a positive integer $d$ such that $d\cdot A_n$ is an algebraic integer for all $n\in \Z$.
    \item $(A_n)_n$ is unital.
\end{enumerate}
\end{corollary}
\begin{proof}
    Suppose $(A_n)_n$ is unital, then
    \[
        A_n=\sum_{\alpha} p_{\alpha}(n)\alpha^n,
    \]
    where the sum runs over all roots of the characteristic polynomial. After multiplication with a sufficiently large integer, we can clear the denominators of the coefficients of the polynomials $p_\alpha(X)$ and get a sequence of algebraic integers. Conversely, suppose that one of the roots is not a unit, then we can find a non-Archimedean place $v$ and a root $\alpha$ such that $|\alpha|_v\neq 1$. Then $(|A_n|_v)_n$ would either be unbounded for $n\to \infty$ or for $n\to -\infty$ by Theorem \ref{EPS}.
\end{proof} 

\begin{theorem}[{\cite[Theorem 2.6]{EPSW}}]\label{EPSW}
Let $\underline{A}=(A_n)_{n\geq 0}$ be a non-degenerate linear recurrence sequence over a number field $K$  and order the roots $\alpha_1,\dots,\alpha_k$ of the characteristic polynomial in such a way that $|\alpha_1|\geq |\alpha_2|\geq \dots \geq |\alpha_k|$. There are positive real numbers $c$ and $C$ such that for any $1\leq H\leq N$ there are at most $C$ values of $n$ with $N-H\leq n\leq N$ such that
\[
|A_n|\leq |\alpha_1|^N \exp(-c\cdot H).
\]
\end{theorem}
\begin{corollary}\label{cor:lower-bound}
	Let $\underline{A}=(A_n)_{n\geq 0}$ be a non-degenerate linear recurrence sequence which has a root of absolute value $1$. Then there exists an $\epsilon>0$ such that
	\[
		|A_n|>\epsilon
	\]
	holds for infinitely many positive integers $n$.
\end{corollary}
\begin{proof}
The assumptions imply $\max\{ |\alpha|: P_{\underline{A}}(\alpha)=0 \}\geq 1$.
Let $H$ be an integer such that $H>C+1$. Set $\epsilon:=\exp(-c\cdot H)$. Theorem \ref{EPSW} implies that for every $N\geq H$ each of the intervals $[N-H,N]$ contains at least one index with $|A_n|>\epsilon$.
\end{proof}

\section{Algebraic approximations to power sums}
The goal of this section is to provide criteria for checking the condition \eqref{eq:Condition_ai} of Proposition \ref{prop1p}. We basically follow the strategy of \cite[Lemma 5]{corv}, however we need a generalization of \cite[Theorem 1]{corv} to power sums in several variables. Such a generalization has been proven by Kulkarni, Mavraki and Nguyen in \cite{kul} using Schmidt Subspace Theorem. In order to state their main theorem, let us introduce the following notation:

Let $h(x)$ denote the absolute logarithmic Weil height and write $\lVert\cdot \rVert$ for the minimal distance to the next integer, i.e. $\lVert x \rVert:=\min_{n\in \mathbb{Z}}\{|x-n|\}$.   By  ${\it sublinear~ function}$, we mean a function $f:\mathbb{N}\rightarrow (0,\infty)$ satisfying $\displaystyle\lim_{n\to\infty}\frac{f(n)}{n}=0$. For a number field $K$, let us write $G_K$ for the absolute Galois group of $K$. Let us now state the main theorem  of Kulkarni, Mavraki and Nguyen, namely Theorem 1.4 of their paper \cite{kul}, though not in its full strength:

\begin{theorem}\label{kmn}{\rm(Kulkarni, Mavraki and Nguyen)}~
Let $r\in\mathbb{N}$, let $(\delta_1,\ldots,\delta_r)$ be a non-degenerate tuple of algebraic numbers  with $|\delta_i|\ge 1$ for $1\leq i\leq r$.  Let $K$ be a number field and  $f$ be a sublinear function.  Suppose for some $\theta\in (0,1)$, the set $\mathcal{M}$ of  tuple  $(n, q_1,\ldots,q_r)\in \mathbb{N}\times(K^\times)^r$ satisfying the inequality 
\begin{equation*}
||q_1 \delta_1^n+\ldots+q_r\delta_r^n||<\theta^n\quad \mbox{and}~~~ \max_{1\leq i\leq k}h(q_i)<f(n)
\end{equation*}
is infinite. Then   the following holds: 
\begin{enumerate}
\item[(i)] $\delta_i$ is an algebraic integer for $i=1,\ldots,r$.
\item[(ii)] For each $\sigma\in G_\mathbb{Q}$  and $1\leq i\leq r$ such that $\frac{\sigma(\delta_i)}{\delta_j}$ is not a root of unity for $1\leq j\leq r$, we have $|\sigma(\delta_i)|<1$.
\end{enumerate}
Moreover for all but finitely many  tuples $(n,q_1,\ldots,q_r)\in\mathcal{M}$
\begin{enumerate}
\item[(iii)]
$\sigma(q_i\delta^n_i)=q_j\delta^n_j$ precisely for  those triples $(\sigma, i, j)\in G_\mathbb{Q}\times\{1,\ldots,r\}^2$ such that $\frac{\sigma(\delta_i)}{\delta_j}$ is  a root of unity. 
\end{enumerate}
\end{theorem}

The following definition is a mild generalization of the notion of being `non-degenerate' for linear recurrence sequences in a number field:
\begin{definition}
	Let $\underline{A}$ be a linear recurrence sequence in a number field $K$ and $L\subseteq K$ a subfield. We call $\underline{A}$ \emph{non-degenerate over $L$} if for any roots $\alpha$ and $\beta$ of $P_{\underline{A}}$ and any $\sigma \in G_L$, we have
\[
	\frac{\sigma(\alpha)}{\beta} \text{ is a root of unity} \Leftrightarrow \sigma(\alpha)=\beta.
\]
\end{definition}
Note that this definition recovers the usual definition of being `non-degenerate' for $L=\Q(c_1,\dots,c_k)$ with $c_1,\dots,c_k$ the coefficients of the characteristic polynomial of $\underline{A}$. \bigskip

From now on, let $K$ denote a real quadratic extension of $\Q$ with a fixed embedding $K\subseteq \R$. In the case of linear recurrences,  Theorem \ref{kmn} specializes to:
\begin{lemma}\label{lemma2}
Let $\underline{A}=(A_n)_{n\geq 0}$ be a non-degenerate linear recurrence sequence in $K$. For an integer $n\geq 0$, write
   $A_n:=[a_0(A_n); a_1(A_n),a_2(A_n),\ldots]$ for the continued fraction expansion. Assume furthermore, that one of the following hypotheses holds:
\begin{enumerate}
\item $P_{\underline{A}}$ has a root of absolute value $\geq 1$ which is not an algebraic integer.
\item $(A_n)_n$ is non-degenerate over $\Q$ and there is a monic irreducible factor $P$ of the characteristic polynomial $P_{\underline{A}-\underline{A}'}$ of $(A_n-A_n')_n$ such that both $P$ and $P'$ have a root of absolute value $\geq 1$.
\end{enumerate}
     Then, for all $i\neq 0$, 
  \begin{equation}\label{eq2.1}
  \displaystyle\lim_{\substack{n\to\infty}}\frac{\log a_i(A_{n})}{n}=0.
  \end{equation}
\end{lemma}

\begin{proof}
We prove the claim by contradiction. Let $m>0$ be the minimum index $i$ such that \eqref{eq2.1} does not hold. Then there exists  $\delta>0$ such that for all $n$ in an infinite set $\mathcal{B}'\subseteq \mathbb{N}$ of positive integers, we have  
\begin{equation}\label{eq2.2}
a_m(A_n)>e^{\delta n}.
\end{equation}
On the other hand, since $\displaystyle\lim_{\substack{n\to\infty}}\frac{\log a_i(A_n)}{n}=0$ for $i=1,\ldots,m-1$, by Lemma \ref{lemma1}, the following holds
\begin{equation}\label{eq:lim_qn_zero}
\lim_{\substack{n\to\infty}}\frac{\log q_{m-1}(A_n)}{n}=0.
\end{equation}
Note that this equation holds for trivial reasons if $m=1$. Recall that, we have
\[
A_n= \sum_{\substack{\alpha\in \C \\ P_{\underline{A}}(\alpha)=0}} p_\alpha(n)\alpha^n.
\]
In view of \eqref{eq1.1} and from \eqref{eq2.2}, we have 
$$
\left|q_{m-1}(A_n)\cdot\left(\sum_{\alpha \text{ root of } P_{\underline{A}}} p_\alpha(n)\alpha^n\right)-p_{m-1}(A_n)\right|<\frac{1}{a_m(A_n) q_{m-1}(A_n)}\leq \frac{1}{a_m(A_n)}<\left(\frac{1}{e^{\delta}}\right)^n
$$
for all sufficiently large values of $n\in\mathcal{B}'$. Let us label the roots of  $P_{\underline{A}}$ by $\alpha_1,\dots,\alpha_k$ in such a way that
\[
	|\alpha_1|\geq|\alpha_2|\geq \dots \geq |\alpha_k|.
\] 
Note that at least one root of $P_{\underline{A}}$ has absolute value $\geq 1$ by the hypotheses. Let us write $k_0$ for the maximal index such that $|\alpha_{k_0}|\geq 1$. We obtain
\[
	|\alpha_1|\geq|\alpha_2|\geq \dots \geq |\alpha_{k_0}|\geq 1 > |\alpha_{k_0+1}|\geq \dots \geq |\alpha_k|.
\]
For $\epsilon>0$ define $\theta:=\max(|\alpha_{k_0+1}|, e^{-\delta})^{(1-\epsilon)}$. By the triangle inequality, we have for $\epsilon>0$ and all sufficiently large $n\in \mathcal{B}'$:
\begin{equation}\label{eq:PowerSum-Approx}
\left|q_{m-1}(A_n)\cdot\left(\sum_{i=1}^{k_0} p_{\alpha_i}(n)\alpha_i^n\right)-p_{m-1}(A_n)\right|< \theta^{n}.
\end{equation}
Applying Theorem \ref{kmn} with $\delta_i=\alpha_i$ for $1\leq i\leq k_0$, and $q_i=q_{m-1}(A_n)p_{\alpha_i}(n)$ for $1\leq i\leq k_0$, gives a contradiction in each of the cases:\\
In the case of assumption $(a)$, we get a contradiction with conclusion $(i)$ of the Theorem \ref{kmn}.
\smallskip

In the case of assumption $(b)$, we will distinguish several cases. The characteristic polynomial of $(A_n-A_n')_n$ divides the product $P_{\underline{A}}P_{\underline{A}'}$, so after possibly interchanging $P$ by $P'$ we may without loss of generality assume that $P$ divides $P_{\underline{A}}$. Let us fix a root $\alpha_0$ of $P$ of absolute value $\geq 1$.\\

\textit{Case 1, ($P=P'$):} In this case, we have $\Q(\alpha)\cap K=\Q$. Hence, we can find a $\sigma\in G_\Q$ such that $\sigma|_K=(\cdot)'$ and $\sigma(\alpha)=\alpha$ for any root $\alpha$ of $P$. Therefore, we get
\begin{equation}\label{eq:AQ-AQprime}
	A_n-A^{\prime}_n=\sum_{\alpha \text{ root of } P_{\underline{A}}} p_{\alpha}(n)\alpha^n-p_{\alpha}^\sigma(n)\sigma(\alpha)^n,
\end{equation}
where $p_{\alpha}^\sigma$ is the polynomial obtained by applying $\sigma$ to the coefficients of $p_{\alpha}$. Since $\alpha_0$ is a root of $(A_n-A'_n)_n$, we deduce that $p_{\alpha_0}(n)\neq\sigma( p_{\alpha_0}(n))$ for infinitely many values of $n$. On the other hand, we have $\sigma(\alpha_0^n)=\alpha_0^n$ by the choice of $\sigma$. Note that $\alpha$ is a root of $P_{\underline{A}}$ and it is of absolute value $\geq 1$, hence it is one of the roots $\alpha_1,\dots,\alpha_{k_0}$ apperaring in equation \eqref{eq:PowerSum-Approx}. We get a contradiction to the conclusion $(iii)$ of the Theorem \ref{kmn}.

\smallskip 

\textit{Case 2, $(P\neq P')$, $P,P'\mid P_{\underline{A}}$:} 
Let $\tilde{\alpha}_0\in \C$ be a root of $P'$ of absolute value $\geq 1$. Since $P\neq P'$, the product $P\cdot P'$ is the minimal polynomial of $\alpha_0$ and $\tilde{\alpha}_0$ over $\Q$ and hence, there is an element $\sigma\in G_\Q$ with $\sigma(\alpha_0)=\tilde{\alpha}_0$. We have $\sigma^{-1}|_K=(\cdot)'$, and hence
\[
    A_n-A_n'=\sum_{\alpha \text{ root of } P_{\underline{A}}} p_{\alpha}(n)\alpha^n-\sigma^{-1}(p_{\alpha}(n))\sigma^{-1}(\alpha)^n.
\]
On the other hand, we can write the sequence $\underline{A}-\underline{A}'$ in its power sum form:
\[
    A_n-A_n'=\sum_{\alpha \text{ root of } P_{\underline{A}-\underline{A}'}} p_{\alpha}(n;\underline{A}-\underline{A}')\alpha^n.
\]
Comparing the two expressions for $A_n-A_n'$ shows
\[
    p_{\alpha_0}(n;\underline{A}-\underline{A}')=p_{\alpha_0}(n;\underline{A})-\sigma^{-1}(p_{\sigma(\alpha_0)}(n;\underline{A})).
\]
On the other hand, we have $p_{\alpha_0}(n;\underline{A}-\underline{A}')\neq 0$, since $\alpha_0$ is a root of $(A_n-A_n')_n$. This shows $\sigma(p_{\alpha_0}(n;\underline{A}))\neq p_{\sigma(\alpha_0)}(n;\underline{A})$. Since both $\alpha_0$ and $\tilde{\alpha}_0$ are roots of $A_n$, they are both among the roots $\alpha_1,\dots, \alpha_{k_0}$ of $(A_n)_n$ appearing in equation \eqref{eq:PowerSum-Approx}. We get a contradiction to the conclusion $(iii)$ of the Theorem \ref{kmn}.

\smallskip

\textit{Case 3, $(P\neq P')$, $P\mid P_{\underline{A}}$ and $P'\nmid P_{\underline{A}}$ :}
Let $\tilde{\alpha}_0\in \C$ be a root of $P'$ of absolute value $\geq 1$. Since $P\neq P'$, the product $P\cdot P'$ is the minimal polynomial of $\alpha_0$ and $\tilde{\alpha}_0$ over $\Q$ and hence, there is an element $\sigma\in G_\Q$ with $\sigma(\alpha_0)=\tilde{\alpha}_0$.
We want to apply criterion $(ii)$ of Theorem \ref{kmn}, we want to check whether there exists a root $\beta$ of $P_{\underline{A}}$ such that  $\frac{\sigma(\alpha_0)}{\beta}$ is a root of unity. Since we assumed that $\underline{A}$ is non-degenerate over $\Q$, this is equivalent to the question whether $\sigma(\alpha_0)$ is a root of $P_{\underline{A}}$. However, if $\sigma(\alpha_0)$ was a root of $P_{\underline{A}}$ then $P'$ would divide $P_{\underline{A}}$ contradicting the assumption in Case 3. Hence, we have shown that there is no root $\beta$ of $P_{\underline{A}}$ such that $\frac{\sigma(\alpha_0)}{\beta}$ is a root of unity.
Conclusion $(ii)$ of the Theorem \ref{kmn} would imply that $|\sigma(\alpha_0)|<1$ in contradiction with the assumption $|\sigma(\alpha_0)|=|\tilde{\alpha}_0|\geq 1$.
\smallskip
\end{proof}

In hypothesis $(b)$ of Lemma \ref{lemma2}, we made the assumption that $(A_n)_n$ is non-degenerate over $\Q$. This assumption is necessary as the following example shows: 
\begin{example}
Take $A_n=\sqrt{2}^n + (1+\sqrt{2})^n$. This sequence is non-degenerate over $K$ but degenerate over $\Q$ since the quotient of the root $\sqrt{2}$ and its conjugate $-\sqrt{2}$ is a non-trivial root of unity. By \cite[Theorem 2]{corv}, it is easily checked that
\[
    A_{2n}=[a_0(A_{2n}); \overline{a_1(A_{2n}),a_2(A_{2n})} ]
\]
has period length $2$ with $a_1(A_{2n})=1$ and $a_2(A_{2n})=\lfloor (1+\sqrt{2})^{2n} \rfloor-1$. In particular,
\[
    \lim_{\substack{n\to\infty}}\frac{\log a_2(A_{n})}{n}\neq 0.
\]
\end{example}
However, we can still weaken the assumption `non-degenerate over $\Q$' by passing to a suitable subsequence:
\begin{corollary}\label{cor:lemma2}
    Let $(A_n)_n$ be a non-degenerate linear recurrence sequence in $K$. Suppose that there is a monic irreducible factor $P$ of $P_{\underline{A}-\underline{A}'}$ such that both $P$ and $P'$ have a root of absolute value $\geq 1$. Then there exist a positive integer $d$ and $0\leq j< d$ such that for all $i\geq 1$
    \[
        \lim_{\substack{n\to\infty}}\frac{\log a_i(A_{dn+j})}{n}= 0.
    \]
\end{corollary}
\begin{proof}
    Let $d$ be a positive integer such that each of the sequences $\underline{A}_{d,j}:=(A_{dn+j})_n$ for $0\leq j<d$ is non-degenerate over $\Q$. We write $\underline{A}-\underline{A}'$ in its power sum form
    \[
    A_n-A_n'=\sum_{\alpha} p_{\alpha}(n;A_{n}-A_{n}')\alpha^n.
    \]
    In particular, we have
    \[
        A_{dn+j}-A_{dn+j}'=\sum_{\alpha} p_{\alpha}(dn+j;\underline{A}-\underline{A}')\alpha^j (\alpha^d)^n.
    \]
    Let $P\in K[X]$ be a monic irreducible divisor of $P_{\underline{A}-\underline{A}'}$ such that both $P$ and $P'$ have a root of absolute value $\geq1$. We write $\alpha_0$ for a root of $P$ of absolute value $\geq 1$. We claim that $\alpha_0^d$ is a root of at least one of the sequences $\underline{A}_{d,j}$ with $0\leq j<d$. Let us write $I:=\{\alpha\in \C: P_{\underline{A}-\underline{A}'}(\alpha)=0 \text{ and }\alpha^d=\alpha_0^d\}$ for the roots $\alpha$ of $\underline{A}-\underline{A}'$ with $d$-th power equal to $\alpha_0^d$. The (possibly vanishing) coefficient of $(\alpha_0^d)^n$ in the power sum expansion of $\underline{A}_{d,j}-\underline{A}_{d,j}'$ is
    \[
        \sum_{\alpha\in I} p_{\alpha}(dn+j;\underline{A}-\underline{A}')\alpha^j.
    \]
    We define the polynomial 
    \[
        p_{\alpha_0^d}(X;\underline{A}_{d,j}-\underline{A}_{d,j}'):=\sum_{\alpha\in I} p_{\alpha}(dX+j;\underline{A}-\underline{A}')\alpha^j.
    \]
    The Vandermonde matrix $(\alpha^j)_{\alpha\in I, 0\leq j<d}$ has full rank and $p_{\alpha}(dX+j;\underline{A}-\underline{A}')$ is a non-zero polynomial for each $\alpha\in I$, hence $p_{\alpha_0^d}(X,\underline{A}_{d,j}-\underline{A}_{d,j}')\neq 0$ for at least one $0\leq j<d$. Since $\alpha_0^d$ is a root of $\underline{A}_{d,j}$ if and only if $p_{\alpha_0^d}(X,\underline{A}_{d,j}-\underline{A}_{d,j}')\neq 0$, we conclude that $\alpha_0^d$ appears as a root of at least one of the sequences $\underline{A}_{d,j}-\underline{A}_{d,j}'$. Fix now one index $j\in\{0,\dots,d-1\}$ such that $\alpha_0^d$ is a root of $\underline{A}_{d,j}-\underline{A}_{d,j}'$. The sequence $\underline{A}_{d,j}$ is non-degenerate over $\Q$ and the minimal polynomial $\tilde{P}$ of $\alpha_0^d$ over $K$ is a monic irreducible polynomial dividing $P_{\underline{A}_{d,j}-\underline{A}_{d,j}'}$ such that both $\tilde{P}$ and $\tilde{P}'$ have roots of absolute value $\geq 1$. We can now apply Lemma \ref{lemma2} to conclude.
\end{proof}

\section{Proof of the main result}
In this section, we will give the proof of our Main Theorem, namely Theorem \ref{thm:A}.
As a preparation, let us make first some basic observations. 
The following lemma implies that sequences satisfying condition $(b)$ of Theorem \ref{thm:A} have bounded period length:
\begin{lemma}\label{lem:bounded-denominators}
    Let $(A_n)_n$ be a sequence with values in $K$ and $(B_n)_n$ be a sequence of rational numbers with bounded denominators, then $(\ell(A_n))_n$ is bounded if and only if $(\ell(A_n+B_n))_n$ is bounded.
\end{lemma}
\begin{proof}
    Let $D$ be a positive integer such that $DB_n\in \Z$ for all $n\in \Z$. We know from Proposition \ref{prop:integer} that the boundedness of the period length is not affected by multiplication by $D$. Now the claim follows from
    \[
        \ell(D (A_n+B_n))=\ell( D A_n).
    \]
\end{proof}
Before we start the proof of the main theorem, we introduce some notation which will be convenient structuring the proof: Let $P\in K[X]$ be a polynomial with factorization into (normalized) irreducible factors $\pi\in K[X]$:
\[
	P=\prod_{\pi\in K[X] \text{ irred.}}\pi^{n_\pi}.
\]
We define
\[
	P^{\Q}:=\prod_{\substack{\pi\in K[X] \text{ irred.}\\ \text{s.t. } \pi|P \text{ and }\pi=\pi'}}\pi^{n_\pi}, \quad P^{K}:=\prod_{\substack{\pi\in K[X] \text{ irred.}\\ \text{s.t. } \pi|P \text{ and }\pi\neq \pi'}}\pi^{n_\pi}.
\]
 In particular, this allows us to factor the characteristic polynomial of a linear recurrence sequence $\underline{A}$:
\[
    P_{\underline{A}}=P^\Q_{\underline{A}}\cdot P^K_{\underline{A}}.
\]
Let us write $\underline{A}$ in its power sum form:
\[
	A_n=\sum_{\alpha \text{ root of } P_{\underline{A}}} p_{\alpha}(n;\underline{A})\alpha^n.
\]
For a linear recurrence $\underline{A}=(A_n)_n$ in $K$, this gives a decomposition $\underline{A}=\underline{A}^{\Q}+\underline{A}^{K}$ of linear recurrence sequences
\[
	A_n=A_n^{\Q}+A_n^{K},
\]
with
\[
	A_n^{\Q}:= \sum_{\alpha \text{ root of } P_{\underline{A}}^{\Q}} p_\alpha(n;\underline{A})\alpha^n, \quad A_n^{K}:= \sum_{\alpha \text{ root of } P_{\underline{A}}^{K}} p_\alpha(n;\underline{A})\alpha^n.
\]

\begin{lemma}\label{lemma3}
Let $(A_n)_n$ be a non-degenerate linear recurrence in $K$ such that $(A^K_n)_n$ is non-zero. The set of roots of $(A^K_n)_n$ does not contain any roots of unity.
\end{lemma}
\begin{proof}
The roots of $P_{\underline{A}}^K$ are a subset of the roots of $(A_n)_n$. Since $(A_n)_n$ is non-degenerate, we conclude that  $(A^K_n)_n$ is also non-degenerate. Hence, the quotient of two distinct roots of $P_{\underline{A}}^K$ is never a root of unity. If $\zeta$ is a root of unity of order $\geq 3$ then $K(\zeta)\neq K$. Hence, if a root of unity $\zeta$ of order $\geq 3$ was a root of $P_{\underline{A}}^K$, we would get a contradiction to the non-degenerateness of $P_{\underline{A}}^K$. It remains to exclude the case $\pm 1$, however in this case, the minimal polynomial $(X\pm 1)$ has coefficients in $\Q$, so $\pm 1$ can never be roots of the characteristic polynomial of $(A^K_n)_n$.
\end{proof}

\begin{proof}[Proof of Theorem \ref{thm:A}]
Let us write $\mathrm{LinRec}_K\subseteq K^\Z$ for the set of non-degenerate recurrence sequences in $K$. Let us write $\mathbf{a},\mathbf{b},\mathbf{c}\subseteq \mathrm{LenRec}_K$ for the three subsets of $\mathrm{LinRec}_K$ corresponding to the cases $(a)$, $(b)$ and $(c)$ in Theorem \ref{thm:A}. More precisely, we define:
\begin{align*}
    \mathbf{a}&:=\mathrm{LinRec}_K\cap \Q^\Z,\\ 
    \mathbf{b}&:=\{\underline{A}\in \mathrm{LinRec}_K: P_{\underline{A}-\underline{A}'}=(X\pm 1)\text{, and }\underline{A}+\underline{A}' \text{ is unital }\},\\ 
    \mathbf{c}&:=\left\{\underline{A}\in \mathrm{LinRec}_K:\begin{matrix} P_{\underline{A}-\underline{A}'}\neq 0 \text{ is unital Pisot, all roots of } \underline{A}+\underline{A}' \text{ of absolute value $\geq 1$ are algebraic integers, }\\
    \text{and all roots of $A+A'$ of absolute value $\leq 1$ are reciprocals of algebraic integers}
    \end{matrix}\right\}. 
\end{align*}

Let us first observe that the sets $\mathbf{a}$, $\mathbf{b}$ and $\mathbf{c}$ are pairwise disjoint: we have $\mathbf{a}\subseteq \Q^\Z$, while $\mathbf{b}\cap \Q^\Z=\mathbf{c}\cap\Q^\Z=\emptyset$. This shows $\mathbf{a}\cap\mathbf{b}=\emptyset$ and $\mathbf{a}\cap\mathbf{c}=\emptyset$. We will prove that $\mathbf{c}$ and $\mathbf{b}$ are disjoint by looking at the characteristic polynomial $P_{\underline{A}-\underline{A'}}$ of $\underline{A}-\underline{A}'$: for $\underline{A}\in \mathbf{b}$, we have $P_{\underline{A}-\underline{A'}}=(X\pm 1)$. For $\underline{A}\in \mathbf{c}$, the polynomial $P_{\underline{A}-\underline{A'}}$ is unital Pisot. Note that the polynomial $(X\pm 1)$ is not unital Pisot. This shows $\mathbf{b}\cap\mathbf{c}=\emptyset$.
\bigskip

Let us now define the following subsets of $\mathrm{LinRec_K}$:
\begin{align*}
    \mathbf{A}&:=\{ \underline{A}\in \mathrm{LinRec}_K: \underline{A}-\underline{A}^{\prime}= 0 \}=\mathrm{LinRec}_K\cap \Q^\Z,\\
    \mathbf{B}&:=\{ \underline{A}\in \mathrm{LinRec}_K: \underline{A}^{\Q}-\underline{A}^{\Q,\prime}\neq 0 \},\\
    \mathbf{C}&:=\{ \underline{A}\in \mathrm{LinRec}_K: \underline{A}^{\Q}-\underline{A}^{\Q,\prime}=0 \text{ and } \underline{A}^{K}-\underline{A}^{K,\prime}\neq 0 \}.
\end{align*}

\smallskip

\textit{Claim:} The sets $\mathbf{A}$, $\mathbf{B}$ and $\mathbf{C}$ are disjoint, we have $\mathbf{A}\cup \mathbf{B}\cup \mathbf{C}=\mathrm{LinRec}_K$ and the inclusions
\[
\mathbf{a}\subseteq \mathbf{A}, \quad \mathbf{b}\subseteq \mathbf{B},\quad \mathbf{c}\subseteq \mathbf{C}.
\]
\textit{Proof of the Claim:}
It is easily seen that we have a disjoint union $\mathbf{A}\cup \mathbf{B}\cup \mathbf{C}=\mathrm{LinRec}_K$ and that $\mathbf{a}=\mathbf{A}$. It remains to show $\mathbf{b}\subseteq \mathbf{B}$ and $\mathbf{c}\subseteq \mathbf{C}$. For a sequence $\underline{A}\in \mathbf{b}$, we have  $\underline{A}-\underline{A}'=\underline{A}^\Q-\underline{A}^{\Q,\prime}$ since the only root of $\underline{A}-\underline{A}'$ is $\pm 1$. This implies $\underline{A}^\Q-\underline{A}^{\Q,\prime}\neq 0$. We deduce $\mathbf{b}\subseteq \mathbf{B}$. Finally, let $\underline{A}\in \mathbf{c}$. The sequence $\underline{A}-\underline{A}'$ is non-zero and unital Pisot. Note that any unital Pisot polynomial $P$ has the property $P=P^K$ and $P^\Q=1$. Applying this to the characteristic polynomial of $\underline{A}-\underline{A}'$ shows 
\[
    \underline{A}-\underline{A}'=\underline{A}^{K}-\underline{A}^{K,\prime},\quad \underline{A}^{\Q}-\underline{A}^{\Q,\prime}=0. 
\]
This implies $\mathbf{c}\subseteq \mathbf{C}$ and concludes the proof of the Claim.\\
\smallskip

The strategy of the proof is now as follows. By the above claim it suffices to show that all sequences of the set
\[
    (\mathbf{A}\setminus \mathbf{a})\cup (\mathbf{B}\setminus \mathbf{b})\cup (\mathbf{C}\setminus \mathbf{c})
\]
have unbounded period length. We will treat the three cases $\mathbf{A}$, $\mathbf{B}$ and $\mathbf{C}$ separately.\\

\smallskip

\textit{Case $\mathbf{A}$:} The set $\mathbf{A}\setminus \mathbf{a}$ is empty, so every sequence in this set has unbounded period length for trivial reasons.\\

\smallskip

\textit{Case $\mathbf{B}$:} Let $\underline{A}=(A_n)_n\in \mathbf{B}$, i.e. $A^\Q_n-A_n^{\Q,\prime}\neq 0$. Our goal is to show that every sequence in $\mathbf{B}\setminus \mathbf{b}$ has unbounded period length. We will prove this in several steps:\\
\smallskip

\textit{Step \textbf{B.1}: If $(A^\Q_n-A^{\Q,\prime}_n)_n$ has a root of absolute value $\neq 1$, then the period length of $(A_n)_n$ is unbounded.} After possibly replacing $(A_n)_n$ by $(A_{-n})_n$, we may assume that  $(A^\Q_n-A^{\Q,\prime}_n)_n$ has a root $\alpha$ with $|\alpha|>1$. Let $P$ be a minimal polynomial of $\alpha$. By the definition of $A^\Q_n$, we have $P=P'$, which in turns implies that  both of $P$ and $P'$ have  the root of  absolute value $>1$. Now, part (b) of Lemma \ref{lemma2} implies that
\[
	\lim_{n\to \infty} \frac{\log a_i(A_n)}{n}=0
\]
for all $i>0$. By Theorem \ref{EPS}, we have
$$
\lim_{n\to\infty}\frac{\log |A_n-A'_n|}{n}>0.
$$
Hence, all the hypotheses of Proposition \ref{prop1p} are satisfied. Now,  part $(a)$ of Proposition \ref{prop1p} implies that the period length is unbounded. This concludes the proof of Step \textbf{B.1}.\\
\smallskip

\textit{Step \textbf{B.2}: If the characteristic polynomial of $(A_n-A_n^{\prime})_n$ is not of the form $(X\pm 1)^m$ then the period length is unbounded.} By Step \textbf{B.1}, the claim follows if  $(A^\Q_n-A^{\Q,\prime}_n)_n$ has a root of absolute value $\neq 1$. Hence, we may assume in the proof of Step \textbf{B.2} that all roots of $(A^\Q_n-A^{\Q,\prime}_n)_n$ have absolute value $1$. Let $\alpha$ be a root of $(A^\Q_n-A^{\Q,\prime}_n)_n$, then its minimal polynomial $P$ divides the characteristic polynomial of $(A_n-A_n')_n$ and satisfies $P=P'$. In particular, $P$ and $P'$ do both have roots of absolute value $1$. Let $D$ be a positive integer, which will be chosen later. Since $(A_n-A_n')_n$ and $(D(A_n-A_n'))_n$ have the same characteristic polynomial, Lemma \ref{lemma2} implies that
\[
	\lim_{n\to \infty} \frac{\log a_i(D A_n)}{n}=0
\]
for all $i>0$. Suppose that $(A_n-A_n')_n$ has a root $\alpha$ which is not a root of unity. Hence, we can find a place $v$ with $|\alpha|_v>1$. Theorem \ref{EPS} implies that 
\[
	\lim_{n\to \infty} \frac{\log|D(A_n-A_n')|_v}{n}>0.
\]
Note that all roots of $(A^\Q_n-A_n^{\Q,\prime})_n$ are also roots of $(A_n-A_n^{\prime})_n$ and hence of $(D(A_n-A'_n))$. +In particular, $(A_n-A_n')_n$ has a root of absolute value $1$.After possibly replacing $(D A_n)_n$ by a subsequence and an appropriate choice of $D$, Corollary \ref{cor:lower-bound} implies that 
\[
    |A_n-A_n'|\geq 2
\]
for all $n\geq 0$. 
 Hence, all the hypotheses  of Proposition \ref{prop1p} are satisfied. Now Proposition \ref{prop1p} and Lemma \ref{lem:periodlength-mult-by-int} implies that the period length is unbounded if $(A_n-A_n')_n$ has a root which is not a root of unity. It remains to study the case where all roots of $(A_n-A_n')_n$ are roots of unity. By non-degeneratedness, we see that the characteristic polynomial of $(A_n-A_n')_n$ for a sequence $(A_n)_n$ with bounded period length has to be of the form $(X-1)^m$ or $(X+1)^m$ for some $m\geq 1$. \\

\textit{Step \textbf{B.3}: If $(A_n+A_n')_n$ is not unital, then the period length is unbounded.} By Step \textbf{B.2}, we may assume that $\pm 1$ is the only root of the characteristic polynomial of $(A_n-A_n')_n$. In particular, for any given positive integer $D$ (chosen later), we have by Lemma \ref{lemma2}:
\[
	\lim_{n\to \infty} \frac{\log a_i(D A_n)}{n}=0 \quad \text{ and }\lim_{n\to \infty} \frac{\log a_i(D A_{-n})}{n}=0.
\]
 By the assumptions of Step \textbf{B.3}, we find a non-Archimedean place $v$ and a root $\alpha$ of $(A_n+A_n')_n$ such that $|\alpha|_v\neq 1$. After possibly replacing $(A_n)_n$ by $(A_{-n})_n$, we may assume $|\alpha|_v>1$. Since $\alpha$ is also a root of $(A_n)_n$ or $(A_n')_n$, it is also a root of $(D A_n)_n$ or $(DA_n')_n$, we deduce
\[
	\lim_n \frac{\log|D A_n|_v}{n}>0 \text{ or }\lim_n \frac{\log|D A_n'|_v}{n}>0.
\] 
After possibly replacing $(D A_n)_n$ by a subsequence, and an appropriate choice of $D$, by  Corollary \ref{cor:lower-bound} we get 
\[
    |D(A_n-A_n')|\geq 2
\]
for all $n\geq 0$. We can now apply Proposition \ref{prop1p}  together with Lemma \ref{lem:periodlength-mult-by-int} to deduce that the sequence $(\ell(A_n))$ of period lengths is unbounded.\\
\smallskip

\textit{Step \textbf{B.4}: If the characteristic polynomial of $(A_n-A_n')_n$ is $(X\pm 1)^m$ with $m\geq 2$ and  $(A_n+A_n')_n$ is unital, then the period length is unbounded.} We write
\[
    A_n=\frac{1}{2}(A_n+A_n')+\frac{1}{2}(A_n-A_n')
\]
and observe that $B_n:=\frac{1}{2}(A_n+A_n')$ is a sequence of rational numbers with bounded denominators by Corollary \ref{cor:nonArch-bounded-units}. Hence, Lemma \ref{lem:bounded-denominators} allows us to reduce to the case $\frac{1}{2}(A_n+A_n')=0$. In this case, the linear recurrence is of the form
\[
	A_n=(\pm 1)^n \sqrt{d_K} P(n),
\]
with $P\in \Q[X]$ with $\deg(P)=m-1$, and $d_K$ the discriminant of the real quadratic field $K$. Now Theorem \ref{schinzel} with $f(X)=d_K\cdot P^2$ implies that the period length of $A_n$ is unbounded.\\
\smallskip

Combining Steps \textbf{B.2}, \textbf{B.3} and \textbf{B.4} shows that every sequence in $\mathbf{B}\setminus \mathbf{b}$ has unbounded period length: Indeed, if $(A_n)_n\in \mathbf{B}\setminus \mathbf{b}$, then either $(A_n+A_n')_n$ is not unital or it is unital but the characteristic polynomial of $(A_n-A_n')_n$ is not of the form $(X\pm 1)$. In the first case, the period length is unbounded by Step \textbf{B.3}, in the second case it is unbounded by Steps \textbf{B.2} and \textbf{B.4}.\\
\smallskip

\textit{Case $\mathbf{C}$:} Let $\underline{A}=(A_n)_n\in \mathbf{C}$, i.e. $A^\Q_n-A_n^{\Q,\prime}= 0$ and $A_n^K-A_n^{K,\prime}\neq 0$. Our goal is to show that every sequence in $\mathbf{C}\setminus \mathbf{c}$ has unbounded period length. We will prove this in several steps:\\

\smallskip

\textit{Step \textbf{C.1}: If there is an irreducible factor $\pi\neq \pi'$ of the characteristic polynomial of $(A_n-A_n')_n$ such that both $\pi$ and $\pi'$ have a root of absolute value $\geq 1$, then the period length is unbounded.} After replacing $(A_n)_n$ by a suitable subsequence of the form $(A_{dn+j})_n$, Corollary \ref{cor:lemma2} implies
\[
	\lim_{n\to \infty} \frac{\log a_i(A_n)}{n}=0.
\]
If there is a root of  $(A_n-A'_n)_n$ of absolute value $>1$, then $(\ell(A_n))_n$ is unbounded by Proposition \ref{prop1p} and Theorem \ref{EPS}. So we may assume that all roots of $(A_n-A_n')_n$ have absolute value $\leq 1$. In this case, we can always find a non-Archimedean place $v$ and a root $\alpha$ of $(A_n-A'_n)_n$ such that $|\alpha|_v>1$. Indeed, suppose we had $|\alpha|_v\leq 1$ for all roots of $(A_n-A'_n)_n$ and all non-Archimedean places $v$, then $\alpha$ would be a root of unity. Such $\alpha$ is either a root of $(A^K_n)$ or $(A^{K,\prime}_n)$, which is a contradiction to Lemma \ref{lemma3}. By Theorem \ref{EPS}, after possibly passing to its conjugate, we have 
$$
\lim_{n\to\infty}\frac{\log|A_n|_v}{n}>0.
$$
Since all roots of $(A_n-A'_n)_n$ are of absolute value $\leq 1$ while there exists a root of $\pi$ of absolute value $\geq 1$, we find at least one root of $(A_n-A'_n)_n$ of absolute value $1$. After possibly replacing $A_n$ by $D A_n$ with an appropriate choice of positive integer $D$ and after passing to a subsequence, by Corollary \ref{cor:lower-bound}, we get
$|A_n-A_n|\geq 2$. 
Now, we can apply part (b) of Proposition \ref{prop1p} together with Lemma \ref{lem:periodlength-mult-by-int} to deduce that the period length of $(A_n)_n$ is unbounded.\\
\smallskip

\textit{Step \textbf{C.2}:  If $(A_n-A'_n)_n$  has a root $\alpha$ of absolute value $1$, then the period length is unbounded.} 
Let $Q$ be the minimal polynomial of $\alpha$, say of degree $m$. By Step \textbf{C.1}, we may assume that all roots of $Q'$ have absolute value $<1$, but then $\pi:=X^m Q(X^{-1})$ is an irreducible factor of the characteristic polynomial $(A_{-n}-A'_{-n})_n$ with root $\alpha^{-1}$ of absolute value $1$. However, all roots of $\pi'$ have absolute value $>1$. By Case \textbf{C.1}, we conclude that $(\ell(A_{-n}))_n$ is unbounded as $n\to \infty$.\\
\smallskip

\textit{Step \textbf{C.3}: If $(A_n-A'_n)_n$ is not Pisot then the period length is unbounded.} This is an immediate consequence of Steps \textbf{C.1} and \textbf{C.2}.\\
\smallskip

\textit{Step \textbf{C.4}: If $(A_n-A'_n)_n$ is Pisot but not unital then the period length is unbounded.} 
Since $(A_n-A_n')_n$ is not unital, there exists a non-Archimedean place $v$ and a root $\alpha$ of $(A_n-A'_n)_n$  such that $|\alpha|_v\neq 1$. After possibly replacing $(A_n-A'_n)_n$ by $(A_{-n}-A'_{-n})_{n}$ we may assume $|\alpha|_v>1$.
Let $\pi$ be the minimal polynomial of $\alpha$. If $\alpha$ has absolute value $<1$, then all roots of $\pi'$ have absolute value $>1$ by the Pisot condition. So after possibly replacing $v$ by $v'$  and $\alpha$ by a root of $\pi'$ there exists a root $\alpha$ of $(A_n-A'_n)_n$ and a non-Archimedean place $v$ such that $|\alpha|_v>1$ and $|\alpha|>1$.  In particular, the characteristic polynomial of $(A_n-A'_n)_n$ has a root of absolute value $>1$ which is not integral and part (a) of Lemma \ref{lemma2} implies that 
\[
	\lim_{n\to \infty} \frac{\log a_i(A_n)}{n}=0.
\]
Since $\alpha$ is also a root of $(A_n-A'_n)_n$, we conclude by Theorem \ref{EPS}
\[
	\lim_{n\to \infty} \frac{\log |A_n-A_n'|}{n}>0.
\]
Now, Proposition \ref{prop1p} shows that the period length is unbounded.\\
\smallskip

\textit{Step \textbf{C.5}: If $(A_n^\Q)_n$  has a root $\alpha$ of absolute value $\geq 1$ which is not an algebraic integer,  but $(A_n-A'_n)$ is unital Pisot, then the period length is unbounded.} By part (a) of Lemma \ref{lemma2}, we have for $i\geq 1$
\[
\lim_{n\to \infty} \frac{\log a_i(A_n)}{n}=0.
\]
Since $(A_n-A'_n)_n$ is unital Pisot, the sequence $(A_n-A_n')_n$ has a root of absolute value $>1$ and Theorem \ref{EPS} implies
\[
\lim_{n\to \infty} \frac{\log|A_n-A_n'|}{n}>0.
\]
Using the similar argument as in previous step, we conclude with Proposition \ref{prop1p}.\\
\smallskip 

\textit{Step \textbf{C.6}: If $(A_n^\Q)_n$  has a root $\alpha$ of absolute value $\leq 1$ which is not the reciprocal of an algebraic integer,  but $(A_n-A'_n)$ is unital Pisot, then the period length is unbounded.} This follows from Step \textbf{C.5} after replacing $(A_n)_n$ by $(A_{-n})_n$.\\

\smallskip 
Combining Steps \textbf{C.1}-\textbf{C.4} shows that a sequence $\underline{A}\in \mathbf{C}$ such that $\underline{A}-\underline{A}'$ is not unital Pisot has unbounded period length. Steps \textbf{C.5}-\textbf{C.6} show that if at least one of the two conditions on $\underline{A}+\underline{A}'$ in $\mathbf{c}$ is violated then the period length is unbounded. Thus we conclude that every sequence in $\mathbf{C}\setminus \mathbf{c}$ has unbounded period length. This concludes the proof of Theorem \ref{thm:A}.
\end{proof}

\section{Unital Pisot sequences}\label{sec:unital-pisot}
As mentioned in Remark 1.1, we expect that there exist sequences of type $(c)$ of Theorem \ref{thm:A} with unbounded period length. In this section, we discuss some examples of the sequences of the form $(c)$  having bounded period length. The simplest example of sequences of type $(c)$ are sequences of the form $A_n=\alpha^n$ with $\alpha$ an irrational unit in the real quadratic field $K$. In this case, by \cite[Theorem 2]{corv}, the sequence $(\ell(\alpha^n))_n$ is bounded. Now, we discuss some further examples of sequences of the form  $(c)$ with bounded period length:
\begin{prop}\label{prop:C} Let $K=\mathbb{Q}(\alpha)$, where $\alpha> 1$ is a quadratic irrational  which is a unit of norm $-1$ in the ring of integers of $K$. Let  $(A_n)_n$   be a sequence in $K$ given by  $A_n=\alpha^{r(n)}+\alpha^{s(n)}$, where $r(n)$  and $s(n)$  are sequences of odd positive integers satisfying the following conditions:
\begin{enumerate}
\item $\psi(n):=s(n)-r(n)> r(n),  ~~ \mbox{for all}~~ n \in \mathbb{N}.$ 
\item $r(n) \to \infty$ as $n \to \infty$. 
\end{enumerate}
Then the sequence $(\ell(A_n))_n$ of  period lengths  is bounded.
\end{prop}
\begin{proof}[Proof of Proposition \ref{prop:C}]
Recall that, we have
\begin{equation*}
    A_n=\alpha^{r(n)}+\alpha^{s(n)}\quad \mbox{for all}~~ n\geq 1.
\end{equation*}
We will prove that the period length of $\underline{A}$ is uniformly bounded by $2$ for all large $n$. For a quadratic irrational $\alpha$, we will use the notation $\mathrm{tr}(\beta)$ and $N(\beta)$ for its trace and norm, respectively. Note that $\alpha$ is a unit of norm  $-1$ , therefore $-1<\alpha'<0.$ Thus, for $n$ sufficiently large, we have $-1<A'_n<0$.
By definition, $\mathrm{tr}(A_n)=A_n+A'_n$ and $N(A_n)=A_nA'_n=-2-\alpha^{\psi(n)}-(\alpha')^{\psi(n)}$.
\bigskip

Let $x_0(n):=A_n$ . Then, for $n$ sufficiently large, $[x_0(n)]=\mathrm{tr}(A_n),$ so we put $a_0(n)=\mathrm{tr}(A_n).$ We have  
$$
x_1(n):=\frac{1}{x_0(n)-[x_0(n)]}=\frac{1}{A_n-\mathrm{tr}(A_n)}=\frac{1}{-A'_n}=\frac{A_n}{-N(A_n)}=\frac{\alpha^{r(n)}+\alpha^{s(n)}}{2+\alpha^{\psi(n)}+(\alpha')^{\psi(n)}}.
$$
Simplifying further, we get:
$$
x_1(n)=\frac{\alpha^{r(n)}}{2+\alpha^{\psi(n)}+(\alpha')^{\psi(n)}}+\frac{\alpha^{s(n)}}{2+\alpha^{\psi(n)}+(\alpha')^{\psi(n)}}=\frac{1}{\frac{2}{\alpha^{r(n)}}+\frac{\alpha^{\psi(n)}}{\alpha^{r(n)}}+\frac{(\alpha')^{\psi(n)}}{\alpha^{r(n)}}}+\frac{\alpha^{r(n)}}{\frac{2}{\alpha^{\psi(n)}}+1+\frac{(\alpha')^{\psi(n)}}{\alpha^{\psi(n)}}}.
$$
Again for $n$ sufficiently large, the integral part of $x_1(n)$ is $[\alpha^{r(n)}]$, so we put $a_1(n)=[\alpha^{r(n)}]$ for $n$ sufficiently large. Since $\alpha$ is a Pisot number and $-1<\alpha'<0$, we have  $[\alpha^{r(n)}]=\mathrm{tr}(\alpha^{r(n)})=\alpha^{r(n)}+(\alpha')^{r(n)}$ and hence $a_1(n)=\mathrm{tr}(\alpha^{r(n)})$.
Then  
$$
x_2(n)=\frac{1}{x_1(n)-a_1(n)}=\frac{1}{\frac{\alpha^{r(n)}+\alpha^{s(n)}}{2+\alpha^{\psi(n)}+(\alpha')^{\psi(n)}}-(\alpha^{r(n)}+(\alpha')^{r(n)})}=\frac{2+\alpha^{\psi(n)}+(\alpha')^{\psi(n)}}{-(\alpha')^{r(n)}-(\alpha')^{s(n)}}.
$$
Now
\begin{align*}
x_2(n)-x_0(n)&=\frac{2+\alpha^{\psi(n)}+(\alpha')^{\psi(n)}}{-(\alpha')^{r(n)}-(\alpha')^{s(n)}}-(\alpha^{r(n)}+\alpha^{s(n)})\\ &=\frac{2+\alpha^{\psi(n)}+(\alpha')^{\psi(n)}+N(\alpha)^{r(n)}-\alpha^{\psi(n)}-(\alpha')^{\psi(n)}+N(\alpha)^{s(n)}}{-(\alpha')^{r(n)}-(\alpha')^{s(n)}}=0
\implies x_2(n)=x_0(n).
\end{align*}
In the last equality, we have used the fact that $\alpha$ is an unit of norm $-1$, and $r(n), s(n)$ are odd positive integers. 
Thus, we concluded that for all sufficiently large $n$, we have the following simple continued fraction representation of $A_n:$
$$
A_n=[\mathrm{tr}(A_n);\overline{[\alpha^{r(n)}],\mathrm{tr}(A_n)}].
$$
This shows that for $n$ sufficiently large, the period length of $A_n$ is $2$, and therefore the sequence $(\ell(A_n))_n$ of  period lengths is bounded. 
\end{proof}

\begin{prop}\label{prop:D}
Let $K=\mathbb{Q}(\alpha)$, where $\alpha> 1$ is a quadratic irrational which is a  unit in the ring of integers of $K$. Let  $(A_n)_n$   be a sequence in $K$  given by $A_n=\alpha^{r(n)}+\alpha^{2 r(n)},$ where $r(n)$ is a sequence of even positive integers such that $r(n) \to \infty$ as $n \to \infty.$ Then  for $n$ sufficiently large  period length of $A_n$ is $4$ and hence the sequence  $(\ell(A_n))_n$ of period lengths is bounded.
\end{prop}
\begin{proof}[Proof of Proposition \ref{prop:D}]
Recall that, we have
\begin{equation*}
    A_n=\alpha^{r(n)}+\alpha^{2 r(n)}\quad \mbox{for all}~~ n\geq 1.
\end{equation*}
We will prove that the period length of $\underline{A}$ is uniformly bounded by $4$ for all large $n$.
 Note that  $\alpha >1$ and $\alpha$ is a unit $,$ therefore either  $-1<\alpha'<0$ or $0<\alpha'<1.$ Since  $r(n)$ is a sequence of even positive integers numbers, we have   $0<(\alpha')^{r(n)}<1$  for each $n\in \mathbb{N}.$ Thus for $n$ sufficiently large $0<A_n'<1.$ By definition, $\mathrm{tr}(A_n)=A_n+A_n'$ and $N(A_n)=A_nA_n'=2+\alpha^{r(n)}+(\alpha')^{r(n)}.$
 \bigskip
 
 Let $x_0(n):=A_n$. Then for $n$ sufficiently large,  $[x_0(n)]= \mathrm{tr}(A_n)-1$, so we put $a_0(n)=\mathrm{tr}(A_n)-1$. We have 
 $$
 x_1(n):=\frac{1}{x_0(n)-[x_0(n)]}=\frac{1}{A_n-(\mathrm{tr}(A_n)-1)}=\frac{1}{1-A_n'}.
 $$
Again  for $n$ sufficiently large, the integral part of $x_1(n)$ is $1$, so we put   $a_1(n)=1.$ Now we find $a_2(n)$. Consider
 $$
 x_2(n):=\frac{1}{x_1(n)-[x_1(n)]}=\frac{1}{\frac{1}{1-A_n'}-1}=\frac{1-A_n'}{A_n'}=\frac{A_n-N(A_n)}{N(A_n)}.
 $$
 Putting the value of $A_n$ and $N(A_n)$ we get:
 $$
 x_2(n)=\frac{\alpha^{r(n)}+\alpha^{2 r(n)}-(2+\alpha^{r(n)}+(\alpha')^{r(n)})}{2+\alpha^{r(n)}+(\alpha')^{r(n)}}=\frac{\alpha^{2r(n)}-2-(\alpha')^{r(n)}}{2+\alpha^{r(n)}+(\alpha')^{r(n)}}.
 $$
 Since $\alpha$ is a Pisot number,  we have  $\alpha^{r(n)}$ is also a Pisot number. Moreover, $0<(\alpha')^{r(n)}<1$ $,$ therefore  $[\alpha^{r(n)}]=\mathrm{tr}(\alpha^{r(n)})-1=\alpha^{r(n)}+(\alpha')^{r(n)}-1.$
 It can be easily seen that  $x_2(n)-1<[\alpha^{r(n)}]-2<x_2(n).$ Thus the integral part of  $[x_2(n)]$ is $[\alpha^{r(n)}]-2$, we put $a_2(n)=[\alpha^{r(n)}]$. Then 
 $$
 x_3(n):=\frac{1}{x_2(n)-[x_2(n)]}=\frac{1}{\frac{\alpha^{2r(n)}-2-(\alpha')^{r(n)}}{2+\alpha^{r(n)}+(\alpha')^{r(n)}}-(\alpha^{r(n)}+(\alpha')^{r(n)}-3)}=\frac{\alpha^{r(n)}+2+(\alpha')^{r(n)}}{2+\alpha^{r(n)}-(\alpha')^{2r(n)}}.
 $$
 For $n$ sufficiently large, $[x_3(n)]=1,$, and we write   $a_3(n)=1$. To calculate partial quotient $a_4(n)$, we consider 
 $$
 x_4(n):=\frac{1}{x_3(n)-[x_3(n)]}=\frac{1}{\frac{\alpha^{r(n)}+2+(\alpha')^{r(n)}}{2+\alpha^{r(n)}-(\alpha')^{2r(n)}}-1}=\frac{2+\alpha^{r(n)}-(\alpha')^{2r(n)}}{(\alpha')^{r(n)}+(\alpha')^{2r(n)}}.
 $$
 Now,
 $$
 x_4(n)-(x_0(n)-1)=\frac{2+\alpha^{r(n)}-(\alpha')^{2r(n)}}{(\alpha')^{r(n)}+(\alpha')^{2r(n)}}-(\alpha^{r(n)}+\alpha^{2r(n)}-1)=0.
 $$
 Therefore  $x_4(n)=x_0(n)-1$ and hence  $[x_4(n)]=[x_0(n)]-1=\mathrm{tr}(A_n)-2.$ We continue the similar process to find the next partial quotient:
 $$
 x_5(n):=\frac{1}{x_4(n)-[x_4(n)]}=\frac{1}{\frac{2+\alpha^{r(n)}-(\alpha')^{2r(n)}}{(\alpha')^{r(n)}+(\alpha')^{2r(n)}}-(\mathrm{tr}(A_n)-2)}.
 $$
 Since $\mathrm{tr(A_n)}=\alpha^{r(n)}+\alpha^{2r(n)}+(\alpha')^{r(n)}+(\alpha')^{2r(n)}$, we get 
 $$
 x_5(n)=\frac{(\alpha')^{r(n)}+(\alpha')^{2r(n)}}{-2(\alpha')^{3r(n)}+(\alpha')^{r(n)}-(\alpha')^{4r(n)}}.
 $$
 Now 
 $$
 x_5(n)-x_1(n)=\frac{(\alpha')^{r(n)}+(\alpha')^{2r(n)}}{-2(\alpha')^{3r(n)}+(\alpha')^{r(n)}-(\alpha')^{4r(n)}}-\left(\frac{1}{1-(\alpha')^{r(n)}-(\alpha')^{2r(n)}}\right)=0,
 $$
 which implies that $x_5(n)=x_1(n)$ and hence $A_n$ has periodic continued fraction representation. Moreover, for all sufficiently large $n$, we have the following simple continued fraction representation of $A_n:$
 $$
 A_n=[\mathrm{tr}(A_n)-1;\overline{1,[\alpha^{p(n)}]-2,1,\mathrm{tr(A_n)}-2}].
 $$
 This shows that for $n$ sufficiently large, the period length of $A_n$ is $4$, and therefore the sequence $(\ell(A_n))_n$ of  period lengths is bounded. 
 \end{proof}

\section*{Appendix A}
In this section, we will prove the following Theorem, which is a more precise version of Theorem \ref{EPS} in the non-Archimedean case:

\begin{theorem}\label{thm1}
Let $K$ be a number field of degree $d$, and $v$ a non-Archimedean place of $K$. Let $(A_n)_n$ be a non-degenerate linear recurrence sequence in $K$ and write $A_n:=a_1(n)\alpha^n_1+\cdots+a_k(n)\alpha^n_k$ with algebraic numbers $\alpha_1,\ldots,\alpha_k$ such that $|\alpha_1|_v\geq |\alpha_2|_v\geq \dots |\alpha_k|_v$ and with non-zero polynomials $a_j(x)\in\overline{\mathbb{Q}}[x]$. Let us write $m$ for the maximum of the degrees of $a_1,\dots, a_k$. Let $S$ be a finite set of places containing the Archimedean places $M_K^\infty$ and such that $\alpha_1,\dots,\alpha_k$ are $S$-units. Suppose that $|\alpha_1|_v>1$. Then for any $\epsilon>0$, we have 
\begin{equation}\label{eq3.4}
|A_n|_\mathit{v}\geq C_8 |a_1(n)|_\mathit{v}|\alpha^n_1|_\mathit{v} n^{-k d m-dm\epsilon}\left(\max_{\substack{j=1,\ldots,k\\ i=1,\ldots,d}}|\sigma_i(\alpha^n_j)|\right)^{-\epsilon d}\cdot \left(\prod_{\omega\in S\backslash M^\infty_K}\max\{|\alpha^n_1|_\omega,\ldots,|\alpha^n_k|_\omega\}\right)^{-\epsilon}
\end{equation}
holds for all sufficiently large values of $n$,  where $\{\sigma_1,\ldots,\sigma_d\}$  is the set of all embeddings of $K$ in $\mathbb{C}$.

\end{theorem}

For a number field $K$, let $M_K$ be the set of all places on $K$ and $M_K^\infty$ be the set of all Archimedean places on $K$. For a non-zero vector ${\bf x}:=(x_0,x_1,\ldots,x_n)\in K^{n+1}$, we define the projective height $H({\bf x})$ as follows:
$$
H({\bf x}):=\displaystyle\prod_{\mathit{v}\in M_K}\max\{|x_0|_\mathit{v},|x_1|_\mathit{v},\ldots,|x_n|_\mathit{v}\}.
$$
We can see that the  projective $H({\bf x})\geq 1$ for all  ${\bf x}\in P^n(K)$.
We also define:
$$
||{\bf x}||:=\max_{\substack{j=0,1,\ldots,t\\ i=1,\ldots,d}}|\sigma_i(x_j)|,
$$
with  $\{\sigma_1,\ldots,\sigma_d\}$  the set of all embeddings of $K$ in $\mathbb{C}$. Moreover, we denote by $\mathcal{O}_{K,S}$  the ring of $S$-integers in $K$. By the product formula, it is easy to see that $\displaystyle\prod_{\mathit{v}\in S}|{\bf x}|_\mathit{v}\geq 1$ for all non-zero ${\bf x}\in\mathcal{O}_{K,s}$.
The following Theorem is required to proof Theorem \ref{thm1}.
\begin{theorem}\label{Evertse}
 Let $t$ be a non-negative integer and $S$ be a finite set of places on $K$, containing all the archimedean places. Then for every $\epsilon>0$, there exists  a constant $C$, depending
only on $\epsilon, S, K, t$ such that for each non-empty subset $T$ of $S$ and every vector ${\mathbf{x}}:=(x_0, x_1,\ldots,x_t)\in\mathcal{O}_{K,S}^{t+1}$, the following holds
\begin{equation}\label{eq3.5}
\left(\prod_{k=0}^t\prod_{\mathit{v}\in S}|x_k|_\mathit{v}\right)\prod_{\mathit{v}\in T}|x_0+x_1+\cdots+x_t|_\mathit{v}\geq C \left(\prod_{\mathit{v}\in T}\max_{0\leq i\leq  t}|x_i|_\mathit{v}\right) H({\bf x})^{-\epsilon}.
\end{equation}
with no  sub-sum  of $x_0+x_1+\cdots+x_t$ vanishes.
\end{theorem}

\noindent{\bf Remark}: Note that when $(x_0, x_1,\ldots,x_t)\in\mathcal{O}_K^{t+1}$, we have
 $$
 H({\bf x})\leq \prod_{\mathit{v}\in M^\infty_K}\max\{|x_0|_\mathit{v},\ldots,|x_t|_\mathit{v}\}\leq ||{\bf x}||^d,
 $$
 where $d:=[K:\mathbb{Q}]$. Then by applying  Theorem \ref{Evertse} with $\epsilon/d$ together with the bound of $H({\bf x})$, we obtain the inequality 
 $$
\left(\prod_{k=0}^t\prod_{\omega\in S}|x_k|_\omega\right)\prod_{\omega\in T}|x_0+x_1+\cdots+x_t|_\omega\geq C \left(\prod_{\omega\in T}\max_{0\leq i\leq  t}|x_i|_\omega\right) ||{\bf x}||^{-\epsilon}.
 $$
 This proves Theorem 8 in \cite{FH}.
\bigskip

Now we give the proof of Theorem \ref{thm1}.
\bigskip

\begin{proof}[Proof of Theorem \ref{thm1}] Let $K$ be a number field of degree $d$ containing the tuple  $(\alpha_1,\ldots,\alpha_k)$ and coefficients of polynomials $a_i(x)$ for $i=1,\ldots,k$. Denote by $L$ its Galois closure. Let $S$ be a suitable finite set of places on $L$ containing all the archimedean places such that $\alpha_i$
is an $S$-unit for each $i=1, 2,\ldots,k$. Let $D$ be a non-zero integer such that $D a_j(n)$ are algebraic integers for all $j=1,\ldots,k$. 
By the hypothesis, there exists a place  $\mathit{v}$ such that $\max\{|\alpha_i|_\mathit{v}:1\leq i\leq k\}>1$. We set $T:=\{\mathit{v}\}$.  Since the tuple $(\alpha_1,\ldots,\alpha_k)$ is non-degenerate, Proposition 2.2 from \cite{kul} yields for all sufficiently large values of $n$, we have 
$$
a_{j_1}(n)\alpha^n_{j_1}+\cdots+a_{j_s}(n)\alpha^n_{j_s}\neq 0
$$
for every non-empty subset $\{j_1,\ldots,j_s\}$ of $\{1,\ldots,k\}$. Thus by Theorem \ref{Evertse}, we get that 
\begin{align*}
\left(\prod_{j=1}^k\prod_{\omega\in S}|D a_j(n)\alpha^n_j|_\omega\right)|D G_n|_\mathit{v}&\geq C\max_{1\leq j\leq k}|D a_j(n)\alpha^n_j|_\mathit{v} H(D {\bf x})^{-\epsilon}\\
&=C\max_{1\leq j\leq t}|D a_j(n)\alpha^n_j|_\mathit{v} H({\bf x})^{-\epsilon},
\end{align*}
where ${\bf x}:=(a_1(n)\alpha^n_1,\ldots,a_k(n)\alpha^n_k)$ and $C$ is some fixed positive real number. Assume that $|\alpha_1|_\mathit{v}:=\displaystyle\max_{1\leq i\leq t}|\alpha_j|_\mathit{v}$. Since $\alpha_j$'s are $S$-units and $D$ is fixed integers, we can re-write the above inequality as
\begin{equation}\label{eq3.6}
\left(\prod_{j=1}^k\prod_{\omega\in S}|a_j(n)|_\omega\right)|G_n|_\mathit{v}\geq C_1| a_1(n)\alpha^n_1|_\mathit{v} H({\bf x})^{-\epsilon},
\end{equation}
for some constant $C_1>0$.
\smallskip

Now we estimate $H({\bf x})$  as follows: by definition
\begin{align*}
H({\bf x})=H(D{\bf x})&=\prod_{\omega\in M_K}\max\{ |D x_1|_\omega,\ldots,|D x_k|_\omega\}\leq \prod_{\omega\in S}\max\{|D x_1|_\omega,\ldots,|D x_k|_\omega\}\\
&= \prod_{\omega\in S}\max\{|D a_1(n)\alpha^n_1|_\omega,\ldots,|D a_k(n)\alpha^n_k|_\omega\}\\
&\leq\displaystyle\prod_{\omega\in M^\infty_K}\max\{|D a_1(n)\alpha^n_1|_\omega,\ldots,|D a_k(n)\alpha^n_k|_\omega|\}\displaystyle\prod_{\omega\in S\backslash M^\infty_K}\max\{|\alpha^n_1|_\omega,\ldots,|\alpha^n_k|_\omega\}.
\end{align*}
In the last line, we have used the fact that $D a_j(n)$ is algebraic integer for each $j=1,2\ldots,k$. We have 
\begin{align*}
H({\bf x})&\leq C_2 n^{dm} \left(\max_{\substack{j=1,\ldots,k\\ i=1,\ldots,d}}|\sigma_i(\alpha^n_j)|\right)^d \prod_{\omega\in S\backslash M^\infty_K}\max\{|\alpha^n_1|_\omega,\ldots,|\alpha^n_k|_\omega\}
\end{align*}
where $m:=\displaystyle\max_{1\leq j\leq k}\mbox{deg}~ a_j(x)$ and some positive constant $C_2$. Substituting this estimate into \eqref{eq3.6}, we get 
$$
\left(\prod_{j=1}^k\prod_{\omega\in S}|a_j(n)|_\omega\right)|G_n|_\mathit{v}\geq C_3 |a_1(n)|_\mathit{v}|\alpha^n_1|_\mathit{v} n^{- dm \epsilon}\left(\max_{\substack{j=1,\ldots,k\\ i=1,\ldots,d}}|\sigma_i(\alpha^n_j)|\right)^{-d\epsilon}\cdot\left(\prod_{\omega\in S\backslash M^\infty_K}\max{|\alpha^n_1|_\omega,\ldots,|\alpha^n_k|_\omega\}}\right)^{-\epsilon}.
$$
Applying Lemma 9 from \cite{FH} to the product in the brackets on the left hand side gives the bound
$$
\prod_{j=1}^k\prod_{\omega\in S}|a_j(n)|_\omega\leq \prod_{j=1}^k C^{(j)}_6 n^{d m}\leq C_7 n^{k d m},
$$
where $C^{(j)}_6$ and $C_7$ are some fixed positive real numbers.
Thus by combining all the above estimates, we have 
$$
|G_n|_\mathit{v}\geq C_8 |a_1(n)|_\mathit{v}|\alpha^n_1|_\mathit{v} n^{-k d m-dm\epsilon}\left(\max_{\substack{j=1,\ldots,k\\ i=1,\ldots,d}}|\sigma_i(\alpha^n_j)|\right)^{-\epsilon d}\cdot \left(\prod_{\omega\in S\backslash M^\infty_K}\max\{|\alpha^n_1|_\omega,\ldots,|\alpha^n_k|_\omega\}\right)^{-\epsilon}.
$$
holds for all $n$ sufficiently large and some fixed $C_8>0$. This completes the proof of Theorem \ref{thm1}.
\end{proof}
\bigskip

We need the following version of Schmidt Subspace Theorem to proof Theorem \ref{Evertse}. For a reference, see (\cite[ Chapter 7]{BG}.
\begin{theorem}\label{schli}{\rm(Subspace Theorem)}~
 Let $K$ be an algebraic number field and $m \geq 2$ an integer. Let $S$ be a finite set of places on $K$ containing all archimedean  places.  For each $v \in S$, let $L_{1,v}, \ldots, L_{m,v}$ be linearly independent linear  forms in the variables $X_1,\ldots,X_m$ with  coefficients in $K$.  For any $\epsilon>0$, the set of solutions $\mathbf{x} \in \mathcal{O}^m_{K,S}$ to the inequality 
\begin{equation*}
\prod_{v\in S}\prod_{i=1}^{m} |L_{i,v}(\mathbf{x})|_v\leq \frac{1}{H(\mathbf{x})^{\epsilon}}
\end{equation*}
lies in finitely many proper subspaces of $K^m$.
\end{theorem}

\begin{proof}[Proof of Theorem \ref{Evertse}] The proof will closely follow the lines of the proof of Theorem 8 in \cite{FH}. The main difference is that we will require the $S$-integer version of the Schmidt subspace theorem, specifically Theorem \ref{schli}. 
\smallskip

Let $S$ be a finite subset of $M_K$, containing all the Archimedean places of $K$, as given in the hypothesis. We shall prove this theorem by induction on $t$. Since 
$$
\prod_{\mathit{v}\in S}|x_0|_\mathit{v}\prod_{\mathit{v}\in T}|x_0|_\mathit{v}\geq \prod_{\mathit{v}\in T}|x_0|_\mathit{v}\cdot 1=\left(\prod_{\mathit{v}\in T}|x_0|_\mathit{v}\right)H(x_0),
$$
in the last equality, we have used that the projective height  $H(x_0)=\displaystyle\prod_{\mathit{v}\in M_K}|x_0|_\mathit{v}=1$. Hence, by taking the constant $C=1$, we are done in the case $t=0$.  Now, suppose that Theorem \ref{Evertse} holds for all $0 \leq t < m$. Our goal is to prove Theorem \ref{Evertse} for $t = m$. Let $\epsilon > 0$, and let $T$ be a non-empty subset of $S$.  We will show that the points ${\bf x}:=(x_0,x_1,\ldots,x_t)\in\mathcal{O}^{t+1}_{K,S}$
which satisfy 
$$
x_{i_0}+x_{i_1}+\cdots+x_{i_s}\neq 0
$$
for each non-empty subset $\{i_0, i_1,\ldots,i_s\}$ of $\{0,1,\ldots,m\}$ and 
\begin{equation}\label{eq3.7}
|x_{i_{0 \mathit{v}}}|_\mathit{v}:=\max\{|x_{i_{1 \mathit{v}}}|_\mathit{v},\ldots,|x_{i_{m\mathit{v}}}|_\mathit{v}\quad\mbox{for all}~~\mathit{v}\in S, 
\end{equation}
where for each $\mathit{v}\in S$, $\{i_{0 \mathit{v}},\ldots,i_{m\mathit{v}}\}$ is a given permutation of $\{0,1,\ldots,m\}$, and 
\begin{equation}\label{eq3.8}
\left(\prod_{k=0}^t\prod_{\mathit{v}\in S}|x_k|_\mathit{v}\right)\prod_{\mathit{v}\in T}|x_0+x_1+\cdots+x_t|_\mathit{v}\leq \left(\prod_{\mathit{v}\in T}|x_{i_{0\mathit{v}}}|_\mathit{v}\right) H({\bf x})^{-\epsilon}
\end{equation}
do also satisfy \eqref{eq3.5} for a certain constant $C$. In order to prove this, we shall apply Theorem \ref{schli}. For $\mathit{v}\in S$, let us   define $m+1$ linearly independent linear  forms in the variables $x_0,x_1,\ldots,x_m$ as follows: for each $v\in T$
$$
L_{0,\mathit{v}}({\bf x}):=x_0+x_1+\cdots+x_m,
$$
and for $\mathit{v}\in S\backslash T$, we define $L_{0,\mathit{v}}({\bf x})=x_{i_{0\mathit{v}}}$. 
 For each $\mathit{v}\in S$ and $1\leq i\leq  m$, we define 
$$
L_{i,\mathit{v}}({\bf x})=x_i.
$$
 Now, we need to calculate the following quantity
\begin{align*}
\prod_{k=0}^m\prod_{\mathit{v}\in S}|L_{k,\mathit{v}}({\bf x})|_\mathit{v}=\left(\prod_{k=0}^m\prod_{\mathit{v}\in S\backslash T}|L_{k,\mathit{v}}({\bf x})|_\mathit{v}\right)\cdot\left(\prod_{k=1}^m\prod_{\mathit{v}\in T}|L_{k,\mathit{v}}({\bf x})|_\mathit{v}\right) \prod_{\mathit{v}\in T}|L_{0,\mathit{v}}({\bf x})|_\mathit{v}.
\end{align*}
Using the above linear forms, we get
\begin{align*}
\prod_{k=0}^m\prod_{\mathit{v}\in S}|L_{i,\mathit{v}}({\bf x})|_\mathit{v}&= \left(\prod_{\mathit{v}\in S\backslash T}|L_{0,\mathit{v}}({\bf x})|_\mathit{v}\right)\left(\prod_{k=1}^m\prod_{\mathit{v}\in S}|x_k|_\mathit{v}\right)\prod_{\mathit{v}\in T}|x_0+x_1+\cdots+x_m|_\mathit{v}\\
&=\left(\prod_{\mathit{v}\in S\backslash T}|x_{i_{0\mathit{v}}}|_\mathit{v}\right)\left(\prod_{k=1}^m\prod_{\mathit{v}\in S}|x_k|_\mathit{v}\right)\prod_{\mathit{v}\in T}|x_0+x_1+\cdots+x_m|_\mathit{v}\\
&=\left(\prod_{k=0}^m\prod_{\mathit{v}\in S}|x_k|_\mathit{v}\right)\left(\prod_{\mathit{v}\in T}|x_{i_{0 \mathit{v}}}|_\mathit{v}\right)^{-1}\prod_{\mathit{v}\in T}|x_0+x_1+\cdots+x_m|_\mathit{v}.
 \end{align*}
 By \eqref{eq3.8}, we get 
 $$
\prod_{k=0}^t\prod_{\mathit{v}\in S}|L_{i,\mathit{v}}({\bf x})|_\mathit{v}=\prod_{k=0}^m\prod_{\mathit{v}\in S}|L_{i,\mathit{v}}({\bf x})|_\mathit{v}\leq \frac{1}{H({\bf x})^\epsilon}.
 $$
 Hence, by Theorem \ref{schli}, all the solutions will lie in finitely many proper subspaces of $K^{n+1}$. For each subspace, it is possible to express some of the variables $x_i$ in terms of the other variables $x_i$. Thus, there exist finitely many tuples $(\beta_{j_0}, \beta_{j_1},\ldots,\beta_{j_u})$ of numbers in $K$, where $0 \leq u <m = n$, such that each solution ${\bf x} \in \mathcal{O}_{K,S}^{n+1}$ satisfies at least one of the relations
\begin{equation}\label{eq3.9}
x_0+x_1+\cdots+x_m=\beta_{j_0} x_{j_0}+\beta_{j_1} x_{j_1}+\cdots+\beta_{j_u} x_{j_u}.
\end{equation}
WLOG, we may assume that no subsums of the right-hand side are equal to zero. We set $\mathcal{A}_1:=\{j_0,j_1,\ldots,j_u\}$, $\mathcal{A}_2:=\{0,1,\ldots,m\}\backslash\mathcal{A}_1$. Let $T_1$ be a subset of $T$ such that $i_{0\mathit{v}}\in\mathcal{A}_1$ and $T_2$ be  subset of $T$ such that  that $i_{0\mathit{v}}\in \mathcal{A}_2$. Let $D'$ be a non-zero integer such that $D' \beta_{j_0},\ldots, D'\beta_{j_u}$ are algebraic integers, and set $z_\ell=D'\beta_{j_\ell} x_{j_\ell}$ for $\ell=0,1,\ldots,u$ and ${\bf z}=(z_0, z_1,\ldots,z_u)$.  First we note that, 
\begin{align*}
\left(\prod_{k=0}^m\prod_{\mathit{v}\in S}|x_k|_\mathit{v}\right)&\prod_{\mathit{v}\in T}|x_0+x_1+\cdots+x_m|_\mathit{v}=\left(\prod_{k\in \mathcal{A}_2}\prod_{\mathit{v}\in S}|x_k|_\mathit{v}\right)\left(\prod_{k\in \mathcal{A}_1}\prod_{\mathit{v}\in S}|x_k|_\mathit{v}\right)\prod_{\mathit{v}\in T}|x_0+x_1+\cdots+x_m|_\mathit{v}.
\end{align*}
By \eqref{eq3.9}, we have
\begin{align*}
\left(\prod_{k=0}^m\prod_{\mathit{v}\in S}|x_k|_\mathit{v}\right)&\prod_{\mathit{v}\in T}|x_0+x_1+\cdots+x_m|_\mathit{v}=\left(\prod_{k\in \mathcal{A}_2}\prod_{\mathit{v}\in S}|x_k|_\mathit{v}\right)\left(\prod_{k\in \mathcal{A}_1}\prod_{\mathit{v}\in S}|x_k|_\mathit{v}\right)\prod_{\mathit{v}\in T}|x_0+x_1+\cdots+x_m|_\mathit{v}\\ &\geq c_3\left(\prod_{k\in \mathcal{A}_2}\prod_{\mathit{v}\in S}|x_k|_\mathit{v}\right)\prod_{\ell=0}^u\prod_{\mathit{v}\in S}|z_\ell|_\mathit{v}\left(\prod_{\mathit{v}\in T}|z_0+z_1+\cdots+z_u|_\mathit{v}\right).
\end{align*}
Now by the induction hypothesis, we get 
\begin{align*}\label{eq3.10}
\left(\prod_{k=0}^m\prod_{\mathit{v}\in S}|x_k|_\mathit{v}\right)&\prod_{\mathit{v}\in T}|x_0+x_1+\cdots+x_m|_\mathit{v}\\ &\geq c_3\left(\prod_{k\in \mathcal{A}_2}\prod_{\mathit{v}\in S}|x_k|_\mathit{v}\right)\prod_{\ell=0}^u\prod_{\mathit{v}\in S}|z_\ell|_\mathit{v}\left(\prod_{\mathit{v}\in T}|z_0+z_1+\cdots+z_u|_\mathit{v}\right)\\
&\geq c_4 \left(\prod_{k\in \mathcal{A}_2}\prod_{\mathit{v}\in S}|x_k|_\mathit{v}\right)\left(\prod_{\mathit{v}\in T
}\max\{|z_0|_\mathit{v},|z_1|_\mathit{v},\ldots,|z_u|_\mathit{v}\}\right) H({\bf z})^{-\epsilon/2}\\
&\geq c_5 \left(\prod_{k\in \mathcal{A}_2}\prod_{\mathit{v}\in S}|x_k|_\mathit{v}\right)\left(\prod_{\mathit{v}\in T}\max_{k\in \mathcal{A}_1}|x_k|_\mathit{v}\right) H({\bf x})^{-\epsilon/2}.\tag{12}
\end{align*}
Since $x_k\in \mathcal{O}_{K,S}$, by the product formula, we have $\displaystyle\prod_{\mathit{v}\in S}|x_k|_\mathit{v}\geq 1$. Thus using  this fact in \eqref{eq3.10}, we obtain
$$
\left(\prod_{k=0}^m\prod_{\mathit{v}\in S}|x_k|_\mathit{v}\right)\prod_{\mathit{v}\in T}|x_0+x_1+\cdots+x_m|_\mathit{v}\geq c_5 \left(\prod_{\mathit{v}\in T}\max_{k\in \mathcal{A}_1}|x_k|_\mathit{v}\right) H({\bf x})^{-\epsilon/2}.
$$
If $T_1=T$, by \eqref{eq3.7}, we have  $\displaystyle\max_{k\in \mathcal{A}_1}|x_k|_\mathit{v}=\max_{0\leq k\leq m}|x_k|_\mathit{v}$, and hence we get  
$$
\left(\prod_{k=0}^m\prod_{\mathit{v}\in S}|x_k|_\mathit{v}\right)\prod_{\mathit{v}\in T}|x_0+x_1+\cdots+x_m|_\mathit{v}\geq c_5 \left(\prod_{\mathit{v}\in T}\max_{0\leq k\leq m}|x_k|_\mathit{v}\right) H({\bf x})^{-\epsilon}.
$$
This satisfies the induction hypothesis and thus establishes the inequality \eqref{eq3.5} in Theorem \ref{Evertse}. Now we consider the case where $T_1$ is a proper subset of $T$. This means that there exists at least one $\mathit{v} \in T$ such that $i_{0\mathit{v}} \in \mathcal{A}_2$ but not in $\mathcal{A}_1$. Let $T_2$ denote the collection of such $\mathit{v} \in T$. By \eqref{eq3.7}, we have
\begin{align*}
\left(\prod_{k\in \mathcal{A}_2}\prod_{\mathit{v}\in S}|x_k|_\mathit{v}\right)&\left(\prod_{\mathit{v}\in T_2}\max_{k\in \mathcal{A}_1}|x_k|_\mathit{v}\right)=\left(\prod_{k\in \mathcal{A}_2}\prod_{\mathit{v}\in S}|x_k|_\mathit{v}\right)\left(\prod_{\mathit{v}\in T_2}|x_{i_{0\mathit{v}}}|_\mathit{v}\right).
\end{align*}
From \eqref{eq3.9}, we get $c|x_{i_{0\mathit{v}}}|_\mathit{v}\geq |(\beta_{j_0}-1)x_{j_0}|_\mathit{v}+(\beta_{j_1}-1)x_{j_1}|_\mathit{v}+\cdots+(\beta_{j_u}-1)x_{j_u}|_\mathit{v}$ for some $c>0$, and hence 
\begin{align*}
\left(\prod_{k\in \mathcal{A}_2}\prod_{\mathit{v}\in S}|x_k|_\mathit{v}\right)&\left(\prod_{\mathit{v}\in T_2}\max_{k\in \mathcal{A}_1}|x_k|_\mathit{v}\right)\geq c_6 \left(\prod_{k\in \mathcal{A}_2}\prod_{\mathit{v}\in S}|x_k|_\mathit{v}\right)\\
&\times\left(\prod_{\mathit{v}\in T_2}|(\beta_{j_0}-1)x_{j_0}|_\mathit{v}+(\beta_{j_1}-1)x_{j_1}|_\mathit{v}+\cdots+(\beta_{j_u}-1)x_{j_u}|_\mathit{v}\right).
\end{align*}
By \eqref{eq3.9}, we have $(\beta_{j_0}-1)x_{j_0}|_\mathit{v}+(\beta_{j_1}-1)x_{j_1}|_\mathit{v}+\cdots+(\beta_{j_u}-1)x_{j_u}=\displaystyle\sum_{k\in \mathcal{A}_2} x_k$. Consequently, we obtain 
\begin{align*}
\left(\prod_{k\in \mathcal{A}_2}\prod_{\mathit{v}\in S}|x_k|_\mathit{v}\right)&\left(\prod_{\mathit{v}\in T_2}\max_{k\in \mathcal{A}_1}|x_k|_\mathit{v}\right)\geq c_6 \left(\prod_{k\in \mathcal{A}_2}\prod_{\mathit{v}\in S}|x_k|_\mathit{v}\right)\prod_{\mathit{v}\in T_2}\left|\sum_{k\in \mathcal{A}_2} x_k\right|_\mathit{v}.
\end{align*}
By the induction hypothesis on the RHS of this inequality, we get that 
\begin{equation*}
  \left(\prod_{k\in \mathcal{A}_2}\prod_{\mathit{v}\in S}|x_k|_\mathit{v}\right)\left(\prod_{\mathit{v}\in T_2}\max_{k\in \mathcal{A}_1}|x_k|_\mathit{v}\right)\geq c_7 \left(\prod_{\mathit{v}\in T_2}\max_{k\in \mathcal{A}_2}|x_k|_\mathit{v}\right) H({\bf x})^{-\epsilon/2}.
\end{equation*}
Together with \eqref{eq3.10} this implies that
\begin{align*}
    \left(\prod_{k=0}^m\prod_{\mathit{v}\in S}|x_k|_\mathit{v}\right)&\prod_{\mathit{v}\in T}|x_0+x_1+\cdots+x_m|_\mathit{v}\geq c_8 \left(\prod_{\mathit{v}\in T_1}\max_{k\in \mathcal{A}_1}|x_k|_\mathit{v}\right) \left(\prod_{\mathit{v}\in T_2}\max_{k\in \mathcal{A}_2}|x_k|_\mathit{v}\right) H({\bf x})^{-\epsilon}\\
    &=c_8 \prod_{\mathit{v}\in T}\max_{0\leq k\leq m}|x_k|_\mathit{v}H({\bf x})^{-\epsilon}.
\end{align*}
This completes the proof of Theorem \ref{Evertse}. 
\end{proof}
\section*{Acknowlegdment} This work was initiated when the first author was visiting University of Duisburg-Essen, Germany in May 2024. He extends heartfelt gratitude to the institute for its’ hospitality and support. The first named author acknowledges ANRF-PMECRG grant ANRF/ECRG/2024/002315/PMS.  The third named author gratefully acknowledges the support through the DFG funded Collaborative Research Center SFB 1085 ’Higher Invariants’.

\end{document}